\documentclass{amsart}
\usepackage[hmargin=3.5cm,vmargin=3.5cm]{geometry}
\usepackage[utf8]{inputenc}
\usepackage{amsmath,
,amssymb
,amsthm
,stmaryrd
, mathtools, mathrsfs
}
\usepackage{graphicx,svg}
\usepackage{enumerate,blindtext}
\usepackage[linktocpage=true]{hyperref}
\usepackage{todonotes, lipsum}
\usepackage[all,cmtip]{xy}

\makeatletter
\g@addto@macro\bfseries{\boldmath}
\makeatother

\newtheorem{theorem}{Theorem}[section]
\newtheorem{lemma}[theorem]{Lemma}

\newtheorem{proposition}[theorem]{Proposition}
\theoremstyle{definition}
\newtheorem{example}[theorem]{Example}
\newtheorem{remark}[theorem]{Remark}
\newtheorem{definition}[theorem]{Definition}

\newcommand{\G}{\ensuremath{\Gamma}}
\newcommand{\g}{\ensuremath{\gamma}}

\newcommand{\C}{\ensuremath{\mathbb{C}}}

\newcommand{\N}{\ensuremath{\mathbb{N}}}

\newcommand{\ov}{\ensuremath{\overline}}
\newcommand{\mf}{\ensuremath{\mathfrak}}
\newcommand{\mc}{\ensuremath{\mathcal}}

\newcommand{\R}{\ensuremath{\mathbb{R}}}
\newcommand{\inv}{\ensuremath{^{-1}}}

\newcommand{\norm}[1]{\left\|#1\right\|}

\renewcommand{\d}[1][t]{\ensuremath{\left.\frac{d}{d#1}\right|_{#1=0}}}
\renewcommand{\Re}{\ensuremath{\operatorname{Re}}}
\renewcommand{\Im}{\ensuremath{\operatorname{Im}}}

\DeclareMathOperator{\Ad}{Ad}
\DeclareMathOperator{\ad}{ad}

\DeclareMathOperator{\res}{Res}

\DeclareMathOperator{\HC}{HC}

\DeclareMathOperator{\pr}{pr}

\DeclareMathOperator{\WF}{WF}

\DeclareMathOperator{\conv}{conv}
\DeclareMathOperator{\spa}{span}
\DeclareMathOperator{\Diff}{Diff}
\DeclareMathOperator{\RT}{RT}

\makeatletter
\newcommand*\bigcdot{\mathpalette\bigcdot@{.5}}
\newcommand*\bigcdot@[2]{\mathbin{\vcenter{\hbox{\scalebox{#2}{$\m@th#1\bullet$}}}}}
\makeatother

\author{Joachim Hilgert, Tobias Weich, and Lasse L. Wolf}
\title{Higher rank quantum-classical correspondence}
\email{hilgert@math.upb.de, weich@math.upb.de, llwolf@math.upb.de}
\begin{document}
 \begin{abstract}
  For a compact Riemannian locally symmetric space $\G\backslash G/K$ of arbitrary rank we determine the location of certain Ruelle-Taylor resonances for the Weyl chamber action. We provide a Weyl-lower bound on an appropriate counting function for the Ruelle-Taylor resonances and establish a spectral gap which is uniform in $\G$ if $G/K$ is irreducible of higher rank. This is achieved by proving a quantum-classical correspondence, i.e. a 1:1-correspondence between horocyclically invariant Ruelle-Taylor resonant states and joint eigenfunctions of the algebra of invariant differential operators on $G/K$.
 \end{abstract}

\maketitle

\section{Introduction}
Ruelle resonances for an Anosov flow provide a fundamental spectral invariant that does not only reflect many important dynamical properties of the flow but also geometric and topological properties of the underlying manifold. Very recently the concept of resonances was extended to higher rank $\R^n$-Anosov actions and led to the notion of \emph{Ruelle-Taylor\footnote{They were named Ruelle-Taylor resonances because the notion of the Taylor spectrum for commuting operators is a crucial ingredient of their definition.} resonances} which were shown to be a discrete subset $\sigma_{\RT}\subset \C^n$ \cite{higherrank}. It was furthermore shown in \cite{higherrank} that the leading resonances (i.e. those with vanishing real part) are related to mixing properties of the considered Anosov action. In particular, it was shown that if the action is weakly mixing in an arbitrary direction of the abelian group $\R^n$, then $0\in \C^n$ is the only leading resonance. Furthermore,  the resonant states at zero give rise to equilibrium measures that share properties of SRB measures of Anosov flows.

Apart from the leading resonances the spectrum of Ruelle-Taylor resonances has so far not been studied if $n\geq 2$. In particular, when $n\geq 2$, it was not known whether there are other resonances than the resonance at zero. Neither was it known whether there is a spectral gap, i.e. whether the real parts of the resonances are bounded away from zero. In this article we shed some light on these questions by examining the Ruelle-Taylor resonances for the class of Weyl chamber flows via harmonic analysis.

Let us briefly introduce the setting: Let $G$ be a real semisimple Lie group with finite center and Iwasawa decomposition $G=KAN$. Let $\mf a$ be the Lie algebra of $A$ and $M$ the centralizer of $A$ in $K$. Then $A$ is isomorphic to $\R ^n $ where $n$ is the real rank of $G$ and acts on $G/M$ from the right. Hence $A$ also acts on the compact manifold $\mathcal M := \G\backslash G/M$, where $\G\leq G$ is a cocompact torsion-free lattice. It can be easily seen that this action is an Anosov action with hyperbolic splitting $T\mathcal M = E_0\oplus E_s \oplus E_u$ which can be described explicitly in terms of associated vector bundles (see Section~\ref{sec:rt_res} for a general definition of Anosov actions and Proposition~\ref{prop:defAnosov} for the description of the hyperbolic splitting for Weyl chamber flows). 
Furthermore, if $\Sigma\subseteq \mf a^\ast$ is the set of restricted roots with simple system $\Pi$ and positive system $\Sigma^+$ then the positive Weyl chamber is given by $\mf a_+=\{H\in \mf a\mid \alpha(H)>0\,\forall\alpha\in\Pi\}$.

The \emph{Ruelle-Taylor resonances} of this Anosov action are defined as follows: For $H\in \mf a $ let $X_H$ be the vector field on $\mathcal M$ defined by the right $A$-action. Then 
\[
\sigma_{\RT}(X) \coloneqq \{\lambda \in \mf a^*_\C \mid \exists u\in \mc D'_{E_u^\ast}(\mathcal M)\setminus \{0\}\colon (X_H + \lambda(H))u=0\,\forall H\in\mf a\},
\]
where $\mc D'_{E_u^\ast}(M)$ is the set of distributions with wavefront set contained in the annihilator $E_u^\ast\subseteq T^\ast \mc M$ of $E_0\oplus E_u$. The distributions $ u\in \mc D'_{E_u^\ast}( M)$ satisfying $ (X_H + \lambda(H))u=0$ for all $H\in\mf a$ are called \emph{resonant states} of $\lambda$ and the dimension of the space of all such distributions is called the multiplicity $m(\lambda)$ of the resonance $\lambda$. It has been shown in \cite{higherrank} that $\sigma_{\RT}(X) \subset \mf a_\C^*$ is discrete and that all resonances have finite multiplicity. It also follows from that work that the real part of the resonances are located in a certain cone  $\ov{{}_- \mf a^\ast}\subset \mf a^*$ which is the negative dual cone of the positive Weyl chamber $\mf a_+$  (see Section~\ref{sec:notation} for a precise definition). 

In this article we will prove that there is a bijection between a certain subset of the Ruelle-Taylor resonant states and certain joint eigenfunctions of the invariant differential operators on the locally symmetric space $\G\backslash G/K$. Before explaining this correspondence in more detail we state two results on the spectrum of Ruelle-Taylor resonances that we can conclude from the correspondence.

The first result says that for any Weyl chamber flow there exist infinitely many Ruelle-Taylor resonances by providing a Weyl-lower bound on an appropriate counting function.
\begin{theorem}\label{thm:Weyl}
 Let $\rho$ be the half sum of the positive restricted roots, $W$ the Weyl group (see Section~\ref{sec:notation} for a precise definition), and for $t>0$ let
 \[
  N(t):=\sum_{\lambda \in \sigma_{\RT}, \Re(\lambda)=-\rho, \|\Im(\lambda)\|\leq t}m(\lambda).
 \]
 Then for $d:=\textup{dim}(G/K)$ 
\[
 N(t) \geq |W|\textup{Vol}(\G\backslash G/K) \left(2\sqrt{\pi}\right)^{-d}\frac{1}{\Gamma(d/2+1)}t^d+ \mathcal O(t^{d-1}).
\]
More generally, let $\Omega \subseteq \mf a^\ast$ be open and bounded such that $\partial \Omega$ has finite $(n-1)$-dimensional Hausdorff measure. Then 
\[\sum_{\lambda \in \sigma_{\RT}, \Re(\lambda)=-\rho, \Im(\lambda)\in t\Omega} m(\lambda) \geq  |W|\textup{Vol}(\G\backslash G/K) \left(2\pi \right)^{-d} \textup{Vol}(\Ad(K)\Omega) t^d +\mc O(t^{d-1}).
\]

\end{theorem}

The second result guarantees a uniform spectral gap.
\begin{theorem}\label{thm:gap}
Let $G$ be a real semisimple Lie group with finite center, then for any cocompact torsion-free discrete subgroup $\G\subset G$ there is a neighborhood $\mc G\subset\mf a^*$  of $0$ such that 
\[
 \sigma_{\RT} \cap (\mc G \times i\mf a^*) = \{0\}.
\]
If $G$ furthermore has Kazhdan's property (T) (e.g. if $G$ is simple of higher rank), then  the spectral gap $\mc G$ can be taken uniformly in $\Gamma$ and only depends on the group $G$.
\end{theorem}

Let us now explain in some detail the spectral correspondence that is the key to the above results: 

We define the space of \emph{first band resonant states} as those resonant states that are in addition horocyclically invariant 
\[
 \textup{Res}^0_X(\lambda):= \{u\in \mathcal D'_{E^*_u}(\mc M), (X_H+\lambda(H))u=0 \text{ and } \mathcal X u= 0 \text{ for all } H\in\mf a\text{ and }\mc X\in C^\infty(\mc M, E_u)\}
\]
and we call a Ruelle-Taylor resonance a \emph{first band resonance} iff $\textup{Res}^0_X(\lambda)\neq 0$. By working with horocycle operators and vector valued Ruelle-Taylor resonances we will be able to show that all resonances with real part in a certain neighborhood of zero are always first band resonances (see Proposition~\ref{prop:horocycle}). As the Weyl chamber flow is generated by mutually commuting Hamilton flows, we consider the set of Ruelle-Taylor resonances as a \emph{classical spectrum}.

Let us briefly describe the \emph{quantum} side: In the rank one case the quantization of the geodesic flow is given by the Laplacian on $G/K$. In the higher rank case we have to consider the algebra of $G$-invariant differential operators on $G/K$ which we denote by $\mathbb D(G/K)$. As an abstract algebra this is a polynomial algebra with $n$ algebraically independent operators, among them the Laplace operator. These operators descend to $\G\backslash G/K$ and we can define the joint eigenspace 
\[
{}^\G E_\lambda=\{f\in C^\infty(\G\backslash G/K)\mid Df =\chi_\lambda(D)f \quad\forall D\in\mathbb D(G/K)\}
\]
where $\chi_\lambda$ is a character of $\mathbb D(G/K)$ parametrized by $\lambda\in\mf a_\C ^\ast/W$ with the Weyl group $W$. Here $\chi_\rho$ is the trivial character (see Section~\ref{sec:harishchandra}). Let $\sigma_Q$ denote the corresponding \emph{quantum spectrum} $\{\lambda\in\mf a_\C^\ast\mid {}^\G E_\lambda\neq \{0\}\}$.

We have the following correspondence between the classical first band resonant states and the joint quantum eigenspace:
\begin{theorem}\label{thm:qcc}
 Let $\lambda\in \mf a_\C^\ast$  be outside the exceptional set $\mc A\coloneqq \{\lambda\in \mf a_\C^\ast\mid \frac {2\langle \lambda+\rho,\alpha\rangle}{\langle\alpha,\alpha\rangle} \in -\N_{>0}$ for some $\alpha\in \Sigma^+\}$. Then there is a bijection between the finite dimensional vector spaces
 \[
  \pi_\ast : \textup{Res}^0_X(\lambda) \to {}^\G E_{-\lambda-\rho}
 \]
where $\pi_\ast$ is the push-forward of distributions along the canonical projection $\pi:\G\backslash G/M\to \G\backslash G/K$.
\end{theorem}
Using this 1:1-correspondence we can then use results about the the quantum spectrum to obtain obstructions and existence results on the Ruelle-Taylor resonances. Notably we use results of Duistermaat-Kolk-Varadarajan \cite{dkv} on the spectrum $\sigma_Q$ but we also deduce refined information on the quantum spectrum. Here we use  $L^p$-bounds for spherical functions obtained from asymptotic expansions \cite{vdBanSchl87}  and $L^p$-bounds for matrix coefficients based on work by Cowling and Oh \cite{cowling, oh2002}.
Theorem~\ref{thm:Weyl} and Theorem~\ref{thm:gap} as stated above give only a rough version of the information on the Ruelle-Taylor resonances that we can actually obtain. As the full results require some further notation we refrain from stating them in the introduction and refer to Theorem~\ref{thm:main}. We also refer to Figure~\ref{fig:first_band_sl3} for a visualization of the structure of first band resonances for the case of $G=SL(3,\R)$.  
 
\textbf{Methods and related results:} 

The key ingredient to the quantum-classical correspondence is that we can in a first step relate the horocyclically invariant first band resonant states with distributional vectors in some principal series representations. Then we can apply the Poisson transform of \cite{KKMOOT} to get a bijection onto the quantum eigenspace ${}^\G E_{-\lambda-\rho}$. The prototype of such a quantum-classical correspondence has been first established by Dyatlov, Faure and Guillarmou \cite{DFG15} in the case of manifolds of constant curvature or in other words for the rank one group $G=SO(n,1)$. Certain central ideas have however already been present for $G=SO(2,1)$ in the works of Flaminio-Forni and Cosentino \cite{FF03, Cos05}. In the rank one setting there exist several generalizations of the quantum classical correspondence of \cite{DFG15} e.g. to convex cocompact manifolds of constant curvature \cite{GHW18, Had20}, general compact locally symmetric spaces of rank one \cite{GHW21} and vector bundles \cite{KW20, KW17}.

Besides the correspondence between the classical Ruelle resonant states and the quantum Laplace eigenvalues there are several other approaches in the literature establishing  exact relations between the Laplace spectrum and the geodesic flow. One approach is to relate the Laplace spectrum to divisors of zeta functions. Such relations have been obtained for rank one locally symmetric spaces on various levels of  generality by Bunke, Olbrich, Patterson and Perry ($G=SO(n,1)$, $\G$ convex cocompact: \cite{BO96, BO99, PP01}, $G$ real rank one, $\G$ cocompact \cite{BO95}).

A third approach to an exact quantum-classical correspondence is to relate the Laplace spectrum to a transfer operator which represents a time discretized dynamics of the geodesic flow. This type of correspondence was notably studied for hyperbolic surfaces with cusps (see \cite{LZ01, BLZ15, BP19} for results for $G=SL(2,\R)$ and $\G$ discrete subgroups of increasing generality). We refer in particular to the expository article \cite{PZ19} and the introduction of \cite{BP19} for a current state of the art of these techniques. A very first step towards generalizations of this approach to higher rank has been recently achieved in \cite{Poh20} for the Weyl chamber flow on products of Schottky surfaces by the construction of symbolic dynamics and transfer operators. 

Note that in \cite{DFG15} not only the first band of Ruelle resonances was related to the Laplace spectrum but a complete band structure has been established and the higher bands could be related to the Laplace spectrum on divergence free symmetric tensors. In the present article we do not study the higher bands. This will presumably be a very hard question for general semisimple groups $G$ (note that in \cite{DFG15} it was crucial at several points that for $G=SO(n,1)$, $N\cong \R^{n-1}$ is abelian). However it might  be tractable for some concrete groups with simple enough root spaces such as $G=SL(3,\R)$. For geodesic flows the phenomenon of such a band structure is quite universal and known in the case of compact locally symmetric spaces of rank one \cite{KW17} but also for geodesic flows on manifolds of pinched negative curvature \cite{FT13,GC20,FT21}.

As mentioned above an important application of Ruelle resonances for Anosov flows are mixing results. More precisely, the existence of a spectral gap in addition with resolvent estimates imply mixing of the flow. For Weyl chamber flows this relation of gaps and mixing rates is not yet established but conjectured to be true. From this perspective Theorem~\ref{thm:gap} is related to the work of
Katok and Spatzier \cite{KS94} who showed exponential mixing for the Weyl chamber action in every direction of the closure of the positive Weyl chamber if $G$ has Property (T). However it is not known whether their result remains true if the Property (T) assumption is dropped. Our result above (Theorem~\ref{thm:gap}) ensures a $\G$-dependent gap in any case but  as mentioned above the precise relation to mixing rates is not yet established.

Finally, Weyl laws for Ruelle resonances of geodesic flows can also be established in variable curvature (or more generally contact Anosov flows)  in various settings \cite{FS11, DDZ14, FT17}. In particular, in the very recent article \cite{FT21} by Faure and Tsujii the Weyl law also follows because a ``first band'' of resonances can be related to a quantum operator. The methods in their work are however completely different and are based on microlocal analysis rather then global harmonic analysis.

{\bf Acknowledgements.} We thank Jan Frahm for discussions about the spherical dual and Erik van den Ban for explanations regarding expansions of spherical functions and Benjamin K\"uster for valuable feedback on the manuscript.
This project has received funding from Deutsche Forschungsgemeinschaft (DFG) (Grant No. WE 6173/1-1  Emmy Noether group ``Microlocal Methods for Hyperbolic Dynamics'').

\section{Preliminaries}

\subsection{Ruelle-Taylor  resonances for higher rank Anosov actions}\label{sec:rt_res}
In this section we recall the main properties of Ruelle-Taylor resonances for higher rank Anosov actions from \cite{higherrank}.
Let $\mc M$ be a compact Riemannian manifold, $A\simeq \R^n$ be an abelian group and let $\tau\colon A\to \operatorname{Diffeo}(\mc M)$ be a smooth locally free group action. If $\mf a\coloneqq \operatorname{Lie}(A)$ we define the generating map $$X\colon\mf a\to C^\infty (\mc M,T\mc M),\quad H\mapsto X_H\coloneqq \d \tau(\exp(tH)).$$ Note that $[X_{H_1},X_{H_2}]=0$ for $H_i\in\mf a$. For $H\in \mf a$ we denote by $\varphi_t^{X_H}$ the flow of the vector field $X_H$. The action is called \emph{Anosov} if there exists $H\in \mf a$ and a continuous $\varphi_t^{X_H}$-invariant splitting $$T\mc M = E_0\oplus E_u\oplus E_s,$$ where $E_0\coloneqq \spa\{X_H\colon H\in\mf a\}$ and there exist $C>0$, $\nu>0$ such that for each $x\in \mc M$
\begin{align*}
 \forall w\in E_s(x), t\geq 0:&\quad \|d\varphi_t^{X_H}(x)w\|\leq Ce^{-\nu t}\|w\|,\\
 \forall w\in E_u(x), t\leq 0:&\quad \|d\varphi_t^{X_H}(x)w\|\leq Ce^{-\nu |t|}\|w\|,
\end{align*}
where the norm on $T\mc M$  is given by the Riemannian metric on $\mc M$.
Such an $H\in \mf a$ is called \emph{transversally hyperbolic}. We call the set $$\mc W \coloneqq \{H'\in \mf a\mid H' \text{ is transversally hyperbolic with the same splitting as } H\}$$ the \emph{positive Weyl chamber containing $H$}. 

Let $E\to \mc M$ be the complexification of a smooth Riemannian vector bundle over $\mc M$ and denote by $\Diff^1(\mc M, E)$ the Lie algebra of first order differential operators with smooth coefficients acting on sections of $E$. Then a Lie algebra homomorphism $\mathbf X\colon \mf a\to \Diff^1(\mc M,E)$ is called an \emph{admissible lift} of the generic map $X$ if 
\begin{equation}\label{eq:admissiblelift}
 \mathbf X_H(fs)=(X_H f)s +f\mathbf X_Hs
\end{equation}
 for $s\in C^\infty(\mc M,E)$, $f\in C^\infty (\mc M)$, and $h\in \mf a$.

For a fixed positive Weyl chamber $\mc W$ the set of \emph{Ruelle-Taylor resonances} can be defined as 
\begin{equation*}
 \sigma_{\RT}(\mathbf X) \coloneqq \{\lambda \in \mf a_\C^\ast \mid \exists u\in \mc D'_{E_u^\ast}(\mc M,E)\setminus \{0\}\colon (\mathbf X_H + \lambda(H))u=0\,\forall H\in\mf a\},
\end{equation*}
where $\mc D'_{E_u^\ast}(\mc M,E)$ is the set of distributional sections of the bundle $E$ with wavefront set contained in $E_u^\ast$. Here $E_u^\ast$ is defined as the annihilator of  $E_0\oplus E_u$ in $T^\ast \mc M$.
The vector space of \emph{Ruelle-Taylor resonant states} for a resonance $\lambda\in \sigma_{\RT}(\mathbf X)$ is defined by 
$$\res_{\mathbf X}(\lambda) \coloneqq \{u\in \mc D'_{E_u^\ast}(\mc M,E)\mid (\mathbf X_H + \lambda(H))u=0\,\forall H\in\mf a\}.$$

\begin{remark}
 The original definition of Ruelle-Taylor resonances and resonant states is stated via Koszul complexes (see \cite[Section~3]{higherrank}). More precisely, $\lambda$ is a resonance iff the corresponding Koszul complex is not exact and the resonant states are the cohomologies of this complex. The space of resonant states that we are considering is just the 0th cohomology. However, it turns out that the Koszul complex is not exact iff the 0th cohomology is non-vanishing, i.e. the two notions coincide (see \cite[Theorem~4]{higherrank}).
\end{remark}

It is known that the resonances have the following properties.
\begin{proposition}[{see \cite[Theorems~1 and 4]{higherrank}}]\label{prop:generalAnosov}
  $\sigma_{\RT}(\mathbf X)$ is a discrete subset of $\mf a_\C^\ast$ contained in $$\{\lambda\in \mf a_\C^\ast\mid \Re (\lambda(H))\leq C_{L^2}(H) \quad \forall H\in \mc W\}$$ with $C_{L^2}(H)= \inf \{C>0 \mid \|e^{-t\mathbf X_H}\|_{L^2\to L^2} \leq e^{Ct} \quad \forall t>0\}$. Moreover, for each $\lambda\in \sigma_{\RT}(\mathbf X)$ the space $\res_{\mathbf X}(\lambda)$ of resonant states is finite dimensional.
\end{proposition}

\subsection{Semisimple Lie groups}\label{sec:notation}
In this section we fix the notation for the present article. Let $G$ be a real semisimple non-compact Lie group  with Iwasawa decomposition $G=KAN$. Furthermore, let $M\coloneqq Z_K(A)$ be the centralizer of $A$ in $K$ and $G=KAN_-$ the opposite Iwasawa decomposition. We denote by $\mf g, \mf a, \mf n,\mf n_-,\mf k,\mf m$ the corresponding Lie algebras. For $g\in G$ let $H(g)$ be the logarithm of the $A$-component in the Iwasawa decomposition.   We have an invariant inner product on $\mf g$ that is induced by the Killing form and the Cartan involution. We have the orthogonal Bruhat decomposition $\mf g =\mf a \oplus \mf m \oplus \bigoplus _{\alpha\in\Sigma} \mf g_\alpha$ into root spaces $\mf g_\alpha$ with respect to the $\mf a$-action via the adjoint action $\ad$. Here $\Sigma\subseteq \mf a^\ast$ is the set of restricted roots. Denote by $W$ the Weyl group of the root system of restricted roots. Let $n$ be the real rank of $G$ and $\Pi=\{\alpha_1,\ldots,\alpha_n\}$ (resp. $\Sigma^+$) the simple (resp. positive) system in $\Sigma$ determined by the choice of the Iwasawa decomposition.  Let $m_\alpha \coloneqq \dim_\R \mf g_\alpha$ and $\rho \coloneqq \frac 12 \Sigma_{\alpha\in \Sigma^+} m_\alpha \alpha$. Denote by $w_0$ the longest Weyl group element, i.e. the unique element in $W$ mapping $\Pi$ to $-\Pi$. Let $\mf a_+ \coloneqq \{H\in \mf a\mid \alpha(H)>0 \,\forall \alpha\in\Pi\}$ the positive Weyl chamber and $\mf a^\ast_+$ the corresponding cone in $\mf a^\ast$ via the identification $\mf a \leftrightarrow \mf a^\ast$ through the Killing form $\langle\cdot,\cdot\rangle$ restricted to $\mf a$. We denote by ${}_+\mf a^\ast$ the dual cone $\{\lambda \in \mf a ^\ast\mid \lambda (H)> 0 \,\forall H\in \ov{\mf a_+}\setminus\{0\}\}$ and by  $\ov{{}_+\mf a^\ast}$ its closure $\{\lambda \in \mf a ^\ast\mid \lambda (H)\geq 0 \,\forall H\in \mf a_+\}=\R_{\geq 0} \Pi$. Hence, if $\omega_j$ is the dual basis of $\alpha_j$ then $\ov{{}_+\mf a^\ast}=\{\lambda\in \mf a^\ast\mid \langle \lambda,\omega_j\rangle\geq 0 \, \forall j=1,\ldots,n\}$. Furthermore, we denote $\ov{ {}_-\mf a^\ast}\coloneqq -\ov{{}_+\mf a ^\ast}$. If $\ov {A^+} \coloneqq \exp (\ov {\mf a_+})$, then we have the Cartan decomposition $G=K\ov {A ^+}K$. 
\begin{example}
 If $G=SL_n(\R)$, then we choose $K=SO(n)$, $A$ as the set of diagonal matrices of positive entries with determinant 1, and $N$ as the set of upper triangular matrices with 1's on the diagonal. $\mf a$ is the abelian Lie algebra of diagonal matrices and the set of restricted roots is $\Sigma= \{\varepsilon_i-\varepsilon_j\mid i\neq j\}$ where $\varepsilon_i(\lambda)$ is the $i$-th diagonal entry of $\lambda$. The positive system corresponding to the Iwasawa decomposition is  $\Sigma^+=\{\varepsilon_i-\varepsilon_j\mid  i<j\}$ with simple system $\Pi=\{\alpha_i = \varepsilon_i-\varepsilon_{i+1}\}$. The positive Weyl chamber is $\mf a_+=\{\operatorname{diag}(\lambda_1,\ldots,\lambda_n)\mid \lambda_1>\cdots>\lambda_n\}$ and the dual cone is $\ov{{}_+ \mf a}=\{\operatorname{diag}(\lambda_1,\ldots,\lambda_n)\in\mf a\mid \lambda_1+\cdots+\lambda_k\geq0\,\forall k\}$. The Weyl group is the symmetric group $S_n$ acting by permutation of the diagonal entries.
\end{example}
\begin{figure}
 \centering
 \includegraphics[height=8cm,trim = 50pt 100pt 50pt 50pt,clip]{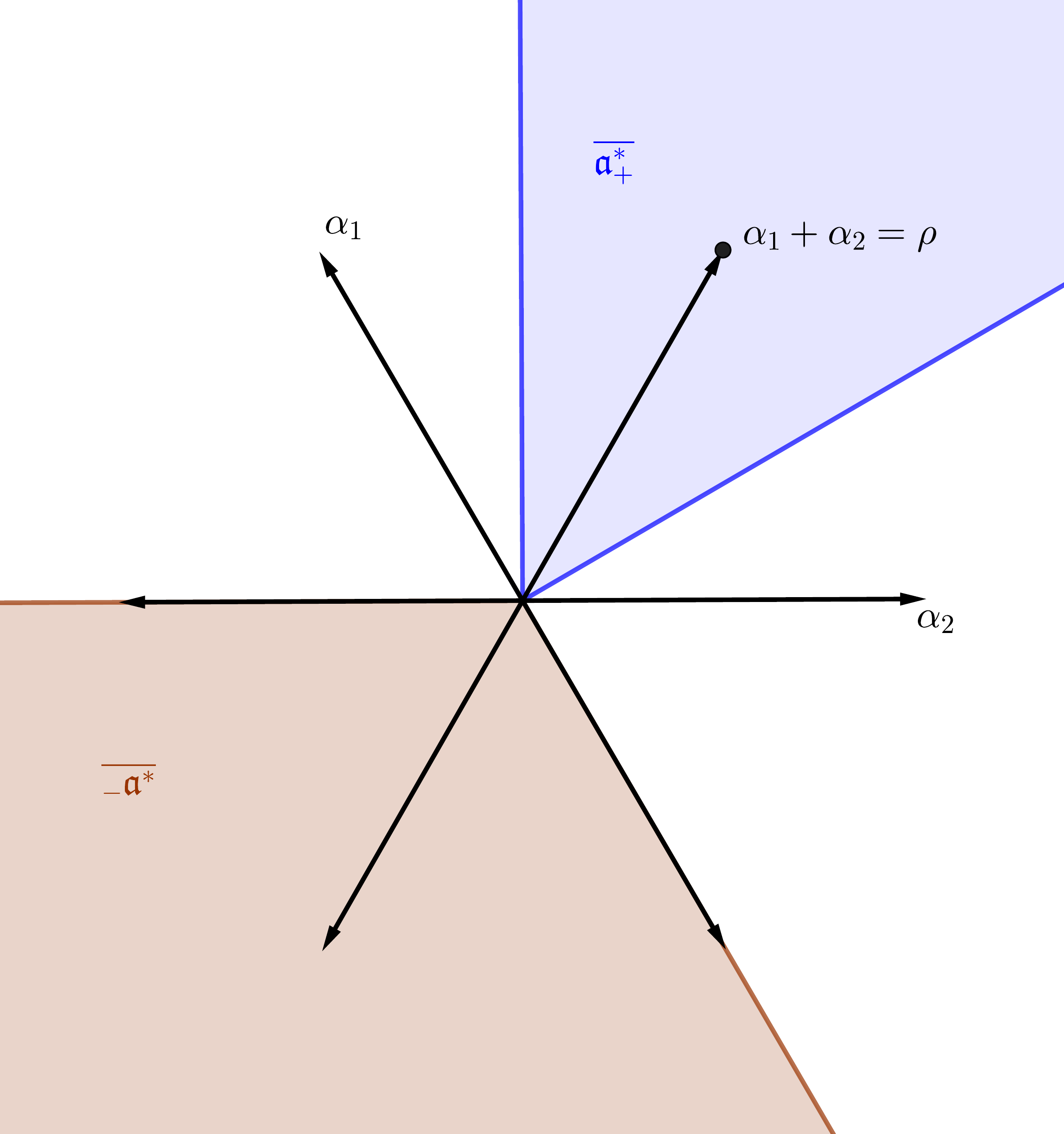}
 \caption{The root system for the special case $G=SL_3(\R)$: There are three positive roots $\Sigma^+=\{\alpha_1,\alpha_2,\alpha_1+\alpha_2\}$. As all root spaces are one dimensional the special element $\rho=\frac 12 \Sigma_{\alpha\in \Sigma^+} m_\alpha \alpha$ equals $\alpha_1+\alpha_2$.}
\end{figure}

\subsection{Principal series representations}\label{sec:principalreps}
The concept of a principal series representation is an important tool in representation theory of semisimple Lie groups. It can be described using different  pictures. We start with the \emph{induced picture}: Pick $\lambda\in \mf a_\C^\ast$ and $(\tau,V_\tau)$ an irreducible unitary representation of $M$. We define $$V^{\tau,\lambda} \coloneqq \{f\colon G\to V_\tau \text{ cont.}\mid f(gman)=e^{-(\lambda+\rho)\log a}\tau(m)\inv f(g), g\in G,m\in M, a\in A,n\in N\}$$ endowed with the norm $\|f\|^2 = \int_K \|f(k)\|^2dk$ where $dk$ is the normalized Haar measure on $K$. The group $G$ acts on $V^{\tau,\lambda}$ by the left regular representation. The completion $H^{\tau,\lambda}$ of $V^{\tau,\lambda}$ with respect to the norm is called \emph{induced picture of the (non-unitary principal series representation} with respect to $(\tau, \lambda)$. We also write $\pi_{\tau,\lambda}$ for this representation. If $\tau$ is the trivial representation then we write $H^\lambda$ and $\pi_\lambda$ and call it the \emph{spherical principal series} with respect to $\lambda$. Note that for equivalent irreducible unitary representations $\tau_1,\tau_2$ of $M$ the corresponding principal series representations are equivalent as representations as well. In particular, the Weyl group $W$ acts on the unitary dual of $M$ by $w\tau (m) = \tau (w\inv mw)$ where $w\in W$ is given by a representative in the normalizer of $A$ in $K$ and therefore $H^{\lambda,w\tau}$ is well-defined up to equivalence.

The \emph{compact picture} is given by restricting the function $f\colon G\to V_\tau$ to $K$, i.e. a dense subspace is given by $$ \{f\colon K\to V_\tau \text{ cont.}\mid f(km)=\tau(m)\inv f(k), k\in K,m\in M\}$$ with the same norm as above. In this picture the $G$-action is given by $$\pi_{\tau,\lambda}(g)f(k)= e^{-(\lambda+\rho)H(g\inv k)}f (k_{KAN} (g\inv k)),\quad g\in G, k\in K,$$ where $k_{KAN}$ is the $K$-component in the Iwasawa decomposition $G=KAN$. 

Recall the associated vector bundle $\mc V_\tau$ over a homogeneous space $L/H$ for a finite dimensional representation $(\tau,V_\tau)$ of $H$.  
Its total space is given by $\mc V_\tau = L\times_\tau V_\tau = (L\times V_\tau)/_\sim$ where $(lh,v)\sim(l,\tau(h)v)$ with $l\in L$, $h\in H$ and $v\in V_\tau$. The equivalence classes are denoted by $[l,v]$ and the projection is $[l,v]\mapsto lH$.  A section $s$ of this bundle can be identified with a function $\ov s\colon L\to V_\tau$ satisfying $\ov s(lh)=\tau(h)\inv \ov s(l)$.  We also have an $L$-action on $\mc V_\tau$ defined by $l' [l,v] \coloneqq [l'l,v]$. 

We can identify the principal series representation $H^{\tau,\lambda}$ with $L^2$-sections of an associated bundle. If $\mc V_\tau^{\mc B}$ denotes the restriction of  the vector bundle $\mc V_\tau$ over $G/M$ to $K/M\subseteq G/M$ we obtain the principal series representation with parameters $(\lambda,\tau)$ as the Hilbert space of  $L^2$-sections of the bundle $\mc V_\tau^{\mc B}$ over $K/M$ with the action $$\ov {\pi_{\tau,\lambda}(g)f}(k)= e^{-(\lambda+\rho)H(g\inv k)}\ov f (k_{KAN} (g\inv k)).$$

\subsection{Invariant differential operators}\label{sec:harishchandra}
Let $\mathbb D(G/K)$ be the algebra of \emph{$G$-invariant differential operators} on $G/K$, i.e. differential operators commuting with the left translation by elements  $g\in G$. Then we have an algebra isomorphism $\HC\colon \mathbb D(G/K)\to \text{Poly}(\mf a^\ast)^W$ from $\mathbb D(G/K)$ to the $W$-invariant complex polynomials on $\mf a^\ast$ which is called \emph{Harish-Chandra homomorphism} (see \cite[Ch.~II~Theorem 5.18]{gaga}). For $\lambda\in \mf a^\ast_\C$ let $\chi_\lambda$ be the character of $\mathbb D(G/K)$ defined by $\chi_\lambda(D)\coloneqq \HC(D)(\lambda)$. Obviously, $\chi_\lambda= \chi_{w\lambda}$ for $w\in W$. Furthermore, the $\chi_\lambda$ exhaust all characters of $\mathbb D(G/K)$ (see \cite[Ch.~III Lemma~3.11]{gaga}). We define the space of joint eigenfunctions $$E_\lambda \coloneqq\{f\in C^\infty(G/K)\mid Df = \chi_\lambda (D) f \,\forall D\in  \mathbb D(G/K)\}.$$ Note that $E_\lambda$ is $G$-invariant. 

\subsection{Poisson transform}\label{sec:Poissontransform}
The representation of $G$ on $E_\lambda$ can be described via the \emph{Poisson transform}: If $(H^{\tau,\lambda})^{-\omega}$ denotes the hyperfunction vectors in the principal series, then the Poisson transform $\mc P_\lambda$ maps $(H^{-\lambda})^{-\omega}$ into $E_\lambda$ $G$-equivariantly. It is given by $\mc P_\lambda f (xK)= \int _{K} f(k)e^{-(\lambda+\rho)H(x\inv k)} dk$ if $f$ is a sufficiently regular function in the compact picture of the principal series. If $f$ is given in the induced picture, then $\mc P_\lambda f(xK)$ simply is  $\int_K f(xk) dk$.

It is important to know for which values of $\lambda\in \mf a_\C^\ast$ the Poisson transform is a bijection.
By \cite{KKMOOT} we have that $\mc P_\lambda\colon (H^{-\lambda})^{-\omega}\to E_\lambda$ is a bijection if \begin{align}\label{eq:assumptionA}
-\frac{2\langle \lambda,\alpha\rangle}{\langle \alpha,\alpha\rangle}\not \in \N_{>0}  \quad \text{for all} \quad \alpha \in \Sigma^+.
\end{align}
In particular, $\mc P_\lambda$ is a bijection if $\Re \lambda \in \ov{\mf  a_+^\ast}$.

\subsection{$L^p$-bounds for elementary spherical functions}
One can show that in each joint eigenspace $E_\lambda$ there is a unique $K$-invariant function which has the value $1$ at the identity (see \cite[Ch.~IV~Corollary 2.3]{gaga}). We denote the corresponding bi-$K$-invariant function on $G$ by $\phi_\lambda$ and call it \emph{elementary spherical function}. Therefore, $\phi_\lambda = \phi_\mu$ iff $\lambda =w\mu$ for some $w\in W$. It is given by the Poisson transform of the constant function with value $1$ in the compact picture, i.e. $\phi_\lambda (g) = \int _K e^{-(\lambda+\rho)H(g\inv k)} dk$.

The aim of this section is to establish the following proposition (see Figure~\ref{fig:convex} for a visualization).
\begin{proposition}\label{prop:Lpequivalence}
 Let $p\in [2,\infty[$. Then the elementary spherical function $\phi_\lambda$ is in $L^{p+\varepsilon}(G)$ (where the $L^p$-space is defined via a Haar measure on $G$) for every $\varepsilon >0$ iff $\Re \lambda \in (1-2p\inv) \conv( W\rho)$ where $\conv(W\rho)$ is the convex hull of the finite set $W\rho$.
\end{proposition}
\begin{proof}

First of all note that we only have to consider $\Re\lambda \in \ov {\mf a_+^\ast}$ since $\phi_\lambda=\phi_{\mu}$ iff $\lambda=w\mu $ for some $w\in W$. In this case $\Re \lambda \in (1-2p\inv) \conv( W\rho)$ is equivalent to $\Re\lambda \in (1-2p\inv)\rho + \ov {{}_-\mf a^\ast}$ (see \cite[Ch.~IV~Lemma~8.3]{gaga}). 
\begin{figure}[ht]
\centering
 \includegraphics[height=10cm,trim = 0pt 30pt  0pt 70pt,clip]{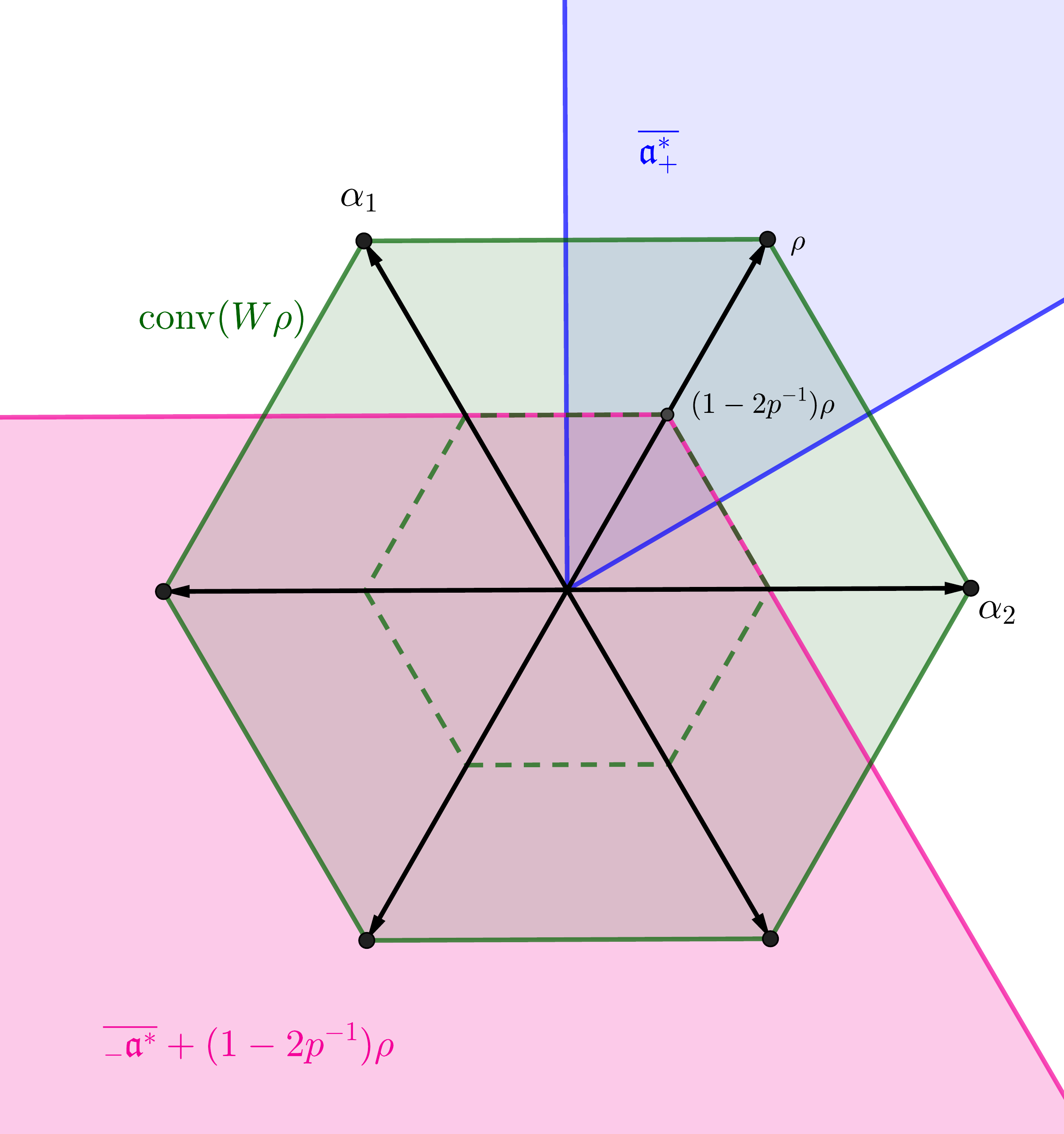}
\caption{Visualization of the regions appearing in Proposition~\ref{prop:Lpequivalence} for the case $G=SL_3(\R)$: The green dashed region is the boundary of $(1-2p\inv) \conv( W\rho)$. Its intersection with the positive Weyl chamber $\ov {\mf a_+^\ast}$ (blue cone) equals $(1-2p\inv)\rho + \ov {{}_-\mf a^\ast}$ intersected with $\ov {\mf a_+^\ast}$.}\label{fig:convex}
\end{figure}

With this remark, one implication of the proposition is a straight forward consequence of standard estimates for elementary spherical functions: Suppose that $\Re\lambda \in \ov {\mf a_+^\ast}$ and $\Re\lambda \in (1-2p\inv)\rho + \ov {{}_-\mf a^\ast}$. Then we have the following bound on $\phi_\lambda$ (see \cite[Ch. VII Prop. 7.15]{Kna86}):
$$|\phi_\lambda(a)|\leq C e^{(\Re \lambda-\rho)(\log a)}(1+\rho(\log a))^d, \quad a\in A^+$$ where $C$ and $d$ are constants $\geq 0$. By the integral formula for $G=K\ov{A^+}K$ (see \cite[Ch. I Theorem 5.8]{gaga}) and the bi-$K$-invariance of $\phi_\lambda$ we have 
\begin{align*}
 \int_G |\phi_\lambda(g)|^{p+\varepsilon}dg &= \int_{\mf a_+} |\phi_\lambda(\exp H)|^{p+\varepsilon} \prod_{\alpha\in \Sigma^+} \sinh(\alpha(H))^{m_\alpha} dH\\
 &\leq \int_{\mf a_+} (C e^{(\Re \lambda-\rho)H}(1+\rho (H))^d)^{p+\varepsilon} e^{2\rho (H)} dH
\end{align*}
for a suitable Lebesgue measure on $\mf a$. Because of $\Re\lambda \in (1-2p\inv)\rho + \ov {{}_-\mf a^\ast}$ we have $$(p+\varepsilon)(\Re\lambda-\rho )(H)\leq -(2+2\varepsilon p\inv)\rho(H).$$
Hence, $$\int_G |\phi_\lambda(g)|^{p+\varepsilon}dg\leq C^{p+\varepsilon} \int_{\mf a_+} (1+\rho (H))^{d(p+\varepsilon)} e^{-2\varepsilon p\inv\rho (H)} dH$$
and we see that the latter is indeed finite by coordinizing $\mf a_+$ by $x_j\leftrightarrow \alpha_j(H)$ with $x_j >0$. Then $dH$ is a multiple of $dx$ and $\rho (H) =\sum x_j \rho_j$ with $\rho_j >0$. Therefore $\phi_\lambda\in L^{p+\varepsilon}(G)$.

The opposite implication  will be proved by combining the proof of \cite[Theorem 8.48]{Kna86}  with  \cite{vdBanSchl87}: According to \cite[Corollary~16.2]{vdBanSchl87} the elementary spherical function $\phi_\lambda$ has a converging expansion 
\begin{equation}
 \phi_\lambda (\exp H)= \sum_{\xi\in X(\lambda)} p_\xi(\lambda,H) e^{\xi(H)}, \quad H\in \mf a_+,
\end{equation}
where $X(\lambda)=\{w\lambda- \rho -\mu\mid w\in W, \mu \in \N_0\Pi\}$ and the $p_\xi(\lambda,\cdot)$ are polynomials of degree $\leq |W|$. The series converges absolutely on $\mf a_+$ and uniformly on each subchamber $\{H\in \mf a_+\mid \alpha_i(H)\geq \varepsilon_i>0\}$.
The main ingredient of the proof of Proposition~\ref{prop:Lpequivalence} is the fact that  (see \cite[Theorem~10.1]{vdBanSchl87})
\begin{equation}
 p_{\lambda-\rho}(\lambda,\cdot) \neq 0.
\end{equation}

Now, if $\phi_\lambda\in L^{p+\epsilon}(G)$, the proof of \cite[Theorem 8.48]{Kna86} shows that $\Re\langle \lambda-(1 -2(p+\epsilon)\inv) \rho,\omega_j\rangle < 0$. Hence $\Re \lambda - (1-2p\inv)\rho\in \ov{ {}_-\mf a^\ast}$.
\end{proof}

\subsection{Positive definite functions and unitary representations}\label{sec:sphericalpos}
In this section we recall the correspondence between positive semidefinite elementary spherical functions and irreducible unitary spherical representations.
Recall first that a continuous function $f\colon G\to \C$ is called \emph{positive semidefinite} if the matrix $(f(x_i\inv x_j))_{i,j}$ for all $x_1,\ldots, x_k\in G$ is positive semidefinite. If $f$ is positive semidefinite, then $f$ is bounded  by $f(1)$ and one has  $f(x\inv)=\ov{ f(x)}$. Moreover, we can define a unitary representation $\pi_f$ associated to $f$ in the following way: If $R$ denotes the right regular representation of $G$, then $\pi_f$ is the completion of the space spanned by $R(x)f$ with respect to the inner product defined by $\langle R(x)f, R(y)f\rangle\coloneqq f(y\inv x)$ which is positive definite. $G$ acts unitarily on this space by the right regular representation. If $f(g)=\langle\pi(g) v,v\rangle$ is a matrix coefficient of a unitary representation $\pi$, then $f$ is positive semidefinite and $\pi_f$ is contained in $\pi$.

Secondly, recall that a unitary representation is called \emph{spherical} if it contains a non-zero $K$-invariant vector. Denote by $\widehat G_{\text{sph}}$ the subset of the unitary dual consisting of spherical representations.   We then have a 1:1-correspondence  between positive semidefinite elementary spherical functions and  $\widehat G_{\text{sph}}$ given by $\phi_\lambda\mapsto \pi_{\phi_\lambda}$  (see \cite[Ch.~IV~Theorem~3.7]{gaga}). The preimage of an irreducible unitary spherical representation $\pi$ with normalized $K$-invariant vector $v_K$ is given by $g\mapsto \langle \pi(g)v_K,v_K\rangle$.  
If the set $\widehat G_{\text{sph}}$ is endowed with the Fell topology (see \cite[Appendix~F.2]{propT}) and we use the topology of convergence on compact sets on the set of elementary spherical functions, then the above correspondence is a homeomorphism as is easily seen from the definitions.

\section{Ruelle-Taylor resonances for the Weyl chamber action}
We keep the notation from Section~\ref{sec:notation}. Let $\G$ be a discrete, torsion-free, cocompact subgroup of $G$. Then the biquotient $\mc M = \G\backslash G/M$ is a smooth compact Riemannian manifold where the Riemannian structure is induced by the inner product on $\mf g$. More precisely, the tangent bundle $T\mc M$  of $\mc M$ is given by the associated vector bundle $\G\backslash (G\times_{\Ad|_M} (\mf a \oplus \mf n \oplus \mf n_-))$  and the norm of some $\G[g,Y], g\in G, Y\in \mf a \oplus \mf n \oplus \mf n_-$ is given by the norm of $Y\in \mf g$. We have a well-defined right $A$-action on $\mc M$:
$$(\G g M )a \coloneqq \G g a M, \quad a\in A, g\in G.$$
Therefore we have an $\mf a$-action by smooth vector fields
$${}_\G X\colon  \mf a \to C^\infty(\mc M,T\mc M), \quad {}_\G X_H f(\G g M) = \d f(\G g e^{tH} M)$$ which we call \emph{Weyl chamber action}.

For later use we denote by $  X\colon\mf a\to \Diff^1 (G/M)$ the corresponding action on $G/M$.

\begin{proposition}\label{prop:defAnosov}
 The $A$-action on $\mc M$ is Anosov. More precisely,  each $H\in \mf a_+$ is transversally hyperbolic with the splitting $E_0 = \G \backslash (G\times_{\Ad|_M} \mf a)$, $E_s = \G \backslash(G\times_{\Ad|_M} \mf n ) $, and  $E_u = \G \backslash (G\times_{\Ad|_M} \mf n_-))$.  Moreover, for fixed $H_0\in \mf a_+$ the dynamically defined positive Weyl chamber  $\mc W =\{H\in \mf a \mid H \text{ is transversally hyperbolic with the same splitting as } H_0\}$ equals $\mf a_+$. Hence the two notions of positive Weyl chambers agree.
 \end{proposition}
 \begin{proof}
  This is immediate from the definitions with the observation that $\mf g_\alpha\perp \mf g_\beta\perp \mf a$ for $\alpha\neq  \beta\neq 0$ in $\Sigma$.
 \end{proof}

\subsection{Lifted Weyl chamber action}
In order to define horocycle operators we  generalize the Weyl chamber action to associated vector bundles. Let $(\tau, V_\tau)$ be a finite-dimensional unitary representation of $M$. Then we have defined the associated vector bundle $\mc V_\tau = G\times_\tau V_\tau$ over $G/M$ (see Section~\ref{sec:notation}). Recall that the total space $\mc V_\tau$ carries the $G$-action $g[g',v]=[gg',v]$ and we have the left regular action on smooth sections of $\mc V_\tau$: $$(gs)(g'M)\coloneqq g (s (g\inv g'M)), \quad s\in C^\infty (G/M,\mc V_\tau).$$
Since $G/M$ is a reductive homogeneous space, i.e. $T(G/M)=G\times_{\Ad|_M} (\mf a \oplus \mf n \oplus \mf n_-)$, we have a  canonical connection $\nabla$ on $\mc V_\tau$ 
given by 
\begin{equation*}
 \nabla_ {{\mf X}} s (gM)= \d \ov s (g\exp(t\ov{\mf X} (g)),
\end{equation*}
where $s$ is a smooth section identified with a smooth function $\ov s\colon G\to V_\tau$ with $\ov s(gm)=\tau(m\inv)\ov s(g)$, $m\in M$, $g\in G$, and $\mf X$ is a vector field  of $G/M$ identified with a smooth function $\ov {\mf X}\colon G\to \mf a\oplus \mf n \oplus \mf n_-$ which is right $M$-equivariant.

The quotient bundle $\G\backslash \mc V_\tau$ is a Riemannian vector bundle over $\mc M$, where the Riemannian structure is induced by the inner product on $V_\tau$. We identify smooth sections $s$ of this bundle with smooth functions $\ov s \colon G\to V_\tau$ with $\ov s(\g gm)=\tau(m\inv) \ov s(g)$ for all $ \g\in \G,$ $g\in G$, and $m\in M$. 

The canonical connection $\nabla$ 
descends to a connection ${}_\G\nabla\colon C^\infty(\mc M,\G\backslash \mc V_\tau)\to C^\infty(\mc M,\G\backslash \mc V_\tau\otimes T^\ast \mc M)$ and we have the following formula:
\begin{equation}\label{eq:connectionformula}
 \ov{{}_\G\nabla  s ({\mf X})}(g)\coloneqq \ov{{}_\G\nabla_ {{\mf X}} s} (g)= \d \ov s (g\exp(t\ov{\mf X} (g)),
\end{equation}
where $s$ is a smooth section identified as above and $\mf X$ is a vector field  of $\mc M$ identified with a smooth function $\ov {\mf X}\colon G\to \mf a\oplus \mf n \oplus \mf n_-$ which is left $\G$-invariant and right $M$-equivariant.

\begin{definition}
 The \emph{lifted Weyl chamber action} is defined as $${}_\G \mathbf X^\tau \colon\mf a\to \Diff^1 (\mc M, \G\backslash \mc V_\tau),  \quad {}_\G \mathbf X^\tau_H \coloneqq {}_\G\nabla_ {{\mf X_H}},$$ where $ {\mf X_H}$ is the vector field identified with the constant mapping $G\to \mf a\subseteq \mf a\oplus \mf n \oplus \mf n_-$, $g\mapsto H$.
\end{definition}
The fact that ${}_\G\nabla$ is a covariant derivative implies that ${}_\G \mathbf X^\tau$ is an admissible lift of the Weyl chamber action in the sense of Equation~\eqref{eq:admissiblelift}. 

For later use we denote by $ \mathbf X^\tau \colon\mf a\to \Diff^1 (G/M,  \mc V_\tau)$ the corresponding action on $G/M$.

We can find a non-trivial tube domain in $\mf a_\C^\ast$ which is independent of $\tau$ and contains all Ruelle-Taylor resonances for the lifted Weyl chamber action.
\begin{proposition}\label{prop:locationgeneral}
 The set of Ruelle-Taylor resonances $\sigma_{\RT}({}_\G \mathbf{X}^\tau)$ is contained in $\ov{{}_-\mf a^ \ast} + i\mf a^\ast$.
\end{proposition}
\begin{proof}

By Proposition~\ref{prop:generalAnosov} we have $$\sigma_{\RT}({}_\G \mathbf{X}^\tau)\subseteq \{\lambda\in \mf a_\C^\ast\mid \Re (\lambda(H))\leq C_{L^2}^\tau(H) \quad \forall H\in \mf a_+\}.$$ Hence, it remains to show that $C_{L^2}^\tau(H)\coloneqq \inf \{C>0 \mid \|e^{-t{}_\G \mathbf{X}^\tau_H}\|_{L^2\to L^2} \leq e^{Ct} \quad \forall t>0\} = 0$ for all $H\in \mf a_+$. We show the stronger statement that $e^{-t{}_\G \mathbf{X}^\tau_H}$ is unitary.

Since $M$ commutes with $A$ we have a well defined action of $A$  on $\G\backslash \mc V_\tau$. It is given by $(\G[g,v])a=\G[ga,v]$. This action gives rise to an $A$-action on sections of the bundle  $\G\backslash \mc V_\tau$ defined via $(af)(x)=f(xa)a\inv$ with $f \in C^\infty(\mc M,\G\backslash \mc V_\tau)$, $x\in \mc M$ and $a\in A$. If we identify $f$ with a equivariant function $\ov f\colon G\to V_\tau$, then $\ov{(af)}(g)=\ov f(ga)$. Let  $d\G  g$ be the normalized right $G$-invariant Radon measure on $\G\backslash G$. Then the $L^2$-norm of $f$ is given by $\|f\|^2_{L^2} = \int_{\G\backslash G} \|\ov f(g)\|_{V_\tau}^2 d\G g$ and it follows that the  $A$-action  continued to $L^2(\mc M, \G\backslash \mc V_\tau)$ is unitary. By definition $e^{-t{}_\G \mathbf{X}^\tau_H} f = \exp(-tH)f $ for $f\in L^2(\mc M,\G\backslash \mc V_\tau)$ and therefore $e^{-t{}_\G \mathbf{X}^\tau_H}$ is unitary.
\end{proof}

\subsection{First band resonances and horocycle operators}\label{sec:horocycle}
In analogy to the rank one setting we make the following definition (see \cite[Definition~2.11]{KW17} and \cite[Definition~3.1]{GHW21} in the scalar case).
\begin{definition}
 We call $\lambda\in\sigma_{\RT}({}_\G \mathbf{X}^\tau)$ a \emph{first band resonance} and write $\lambda\in\sigma_{\RT}^0({}_\G \mathbf{X}^\tau)$ if the vector space $$\res_{{}_\G \mathbf{X}^\tau}^0(\lambda) = \{u\in \res_{{}_\G \mathbf{X}^\tau} (\lambda)\mid {}_\G \nabla _{\mf X} u = 0 \, \forall \mf X \in C^\infty(\mc M, E_u)\}$$ of \emph{first band resonant states} is non-trivial. 
\end{definition}
The goal of this section is to prove that in a certain neighborhood of $0$ in $\mf a_\C^\ast$ each Ruelle-Taylor resonance is a first band resonance and $\res_{{}_\G \mathbf{X}^\tau}^0(\lambda) =  \res_{{}_\G \mathbf{X}^\tau} (\lambda)$. This will be done by introducing so called horocycle operators as follows. 

Recall that $T\mc M = \G\backslash(G\times_{\Ad|_M} \mf a\oplus \mf n\oplus \mf n_-)$ and the bundle  $\G\backslash(G\times_{\Ad|_M} \mf n)$ decomposes as $\bigoplus _{\alpha\in \Sigma^+} \G\backslash (G\times _{\Ad|_M} \mf g_\alpha)$, and similarly for $\mf n_-$. Therefore, the cotangent bundle $T^\ast\mc M$ is the Whitney sum $\G\backslash(G\times_{\Ad^\ast|_M} \mf a^\ast )\oplus\bigoplus_{\alpha\in \Sigma} \G\backslash(G\times_{\Ad^\ast|_M} \mf g_\alpha^\ast)$. Let us denote the coadjoint action of $M$ on the complexification of $\mf g_\alpha^\ast$ by $\tau_\alpha$. Note that $\tau_\alpha$ is unitary with respect to the inner product induced by the Killing form. We can now define $$\pr_\alpha\colon (T^\ast\mc M)_\C \to \G\backslash \mc V_{\tau_\alpha}$$ by fiber-wise restriction to the subbundle $\G\backslash (G\times _{\Ad|_M} \mf g_\alpha)$. This induces a map $$\widetilde \pr _\alpha\colon C^\infty (\mc M,\G\backslash\mc V_\tau\otimes (T^\ast\mc M)_\C )\to C^\infty (\mc M, \G\backslash \mc V_{\tau\otimes \tau_\alpha}).$$
\begin{definition}
 If ${}_\G\nabla ^ \C\colon C^\infty(\mc M,\G\backslash \mc V_\tau)\to C^\infty(\mc M,\G\backslash \mc V_\tau\otimes (T^\ast \mc M)_\C)$ denotes the complexification of the canonical connection ${}_\G\nabla $, then the \emph{horocycle operator} $\mc U_\alpha$ for $\alpha\in \Sigma$ is defined as the composition $$\mc U_\alpha \coloneqq \widetilde \pr_\alpha \circ {}_\G\nabla ^ \C \colon  C^\infty(\mc M,\G\backslash \mc V_\tau)\to C^\infty (\mc M, \G\backslash \mc V_{\tau\otimes \tau_\alpha}).$$
\end{definition}
Note that we have the explicit formula 
\begin{equation} \label{eq:horocycleformula}
 \ov {\mc U_\alpha s } (g)(Y) =\d \ov s (g\exp(tY)), \quad s\in C^\infty (\mc M,\G\backslash \mc V_\tau), Y\in \mf g_\alpha,
\end{equation}
 if we again use the identification of  sections of some associated vector bundle with left $\G$-invariant and right $M$-equivariant functions  indicated by $\ov \cdot$ and the identification $V_\tau\otimes \mf g_\alpha^\ast \simeq \mathrm{Hom}(\mf g_\alpha,V_\tau)$.

 We should point out that the space of first band resonant states can be rewritten with the horocycle operators as 
 \begin{equation}
  \res_{{}_\G \mathbf{X}^\tau}^0(\lambda) = \{u\in \res_{{}_\G \mathbf{X}^\tau} (\lambda)\mid \mc U_{-\alpha} u = 0 \,\, \forall \alpha \in \Sigma^+\}.
 \end{equation}
Note that in the case of constant curvature manifolds (i.e. the real hyperbolic case $G=PSO(n,1)$ of rank 1) there is only one positive root and our definition reduces to the original one due to Dyatlov and Zworski (see \cite[p.~931]{DFG15}). Furthermore, our definition extends the definition of the horocycle operators  for arbitrary $G$  of rank one (see \cite{KW17}). 

The horocycle operators fulfill the following important commutation relation.
\begin{lemma}\label{la:commutationhoro}
 $$ \forall H\in \mf a\colon \quad {}_\G \mathbf{X}^{\tau\otimes \tau_\alpha}_H \mc U_\alpha - \mc U_\alpha {}_\G \mathbf{X}^\tau_H = \alpha (H) \mc U_\alpha .$$
\end{lemma}
\begin{proof}
 Using the formulas \eqref{eq:connectionformula} and \eqref{eq:horocycleformula} we obtain 
 $$\ov{{}_\G \mathbf{X}^{\tau\otimes \tau_\alpha}_H \mc U_\alpha - \mc U_\alpha {}_\G \mathbf{X}^\tau_H} (g)(Y) = \d[t_1]\d[t_2] \ov s (g\exp(t_1 H)\exp (t_2 Y)) -  \ov s (g\exp(t_1Y)\exp (t_2H))$$
 and the latter equals $$\d \ov s(g\exp(t[H,Y])).$$ Since $[H,Y]=\alpha(H)Y$ for $Y\in \mf g_\alpha$ the claim follows.
\end{proof}

We can now prove the main result of this section.
\begin{proposition}\label{prop:horocycle}
 The horocycle operators can be extended continuously as linear operators to distributional  sections, i.e. $$\mc U_\alpha\colon \mc D'(\mc M,\G\backslash \mc V_\tau)\to \mc D'(\mc M, \G\backslash \mc V_{\tau\otimes \tau_\alpha}).$$ Moreover, for $\lambda\in\sigma_{\RT}({}_\G \mathbf{X}^\tau)$ the horocycle operator $\,\mc U_{-\alpha}$ maps $$\res_{{}_\G \mathbf{X}^\tau}(\lambda)\quad  \text{into} \quad \res_{{}_\G \mathbf{X}^{\tau\otimes \tau_{-\alpha}}}(\lambda + \alpha).$$
 In particular, each  $\lambda\in\sigma_{\RT}({}_\G \mathbf{X}^\tau)$ with $\Re \lambda\in \bigcap_{\alpha\in \Pi} \ov {{}_-\mf a^\ast} \setminus (\ov {{}_-\mf a^\ast}-\alpha)$  is a first band resonance and $\res_{{}_\G \mathbf{X}^\tau}(\lambda)=\res_{{}_\G \mathbf{X}^\tau}^0(\lambda)$ holds.
\end{proposition}

\begin{figure}[ht]
 \centering
 \includegraphics[height=10cm, trim=80pt 100pt 50pt 150 pt,clip]{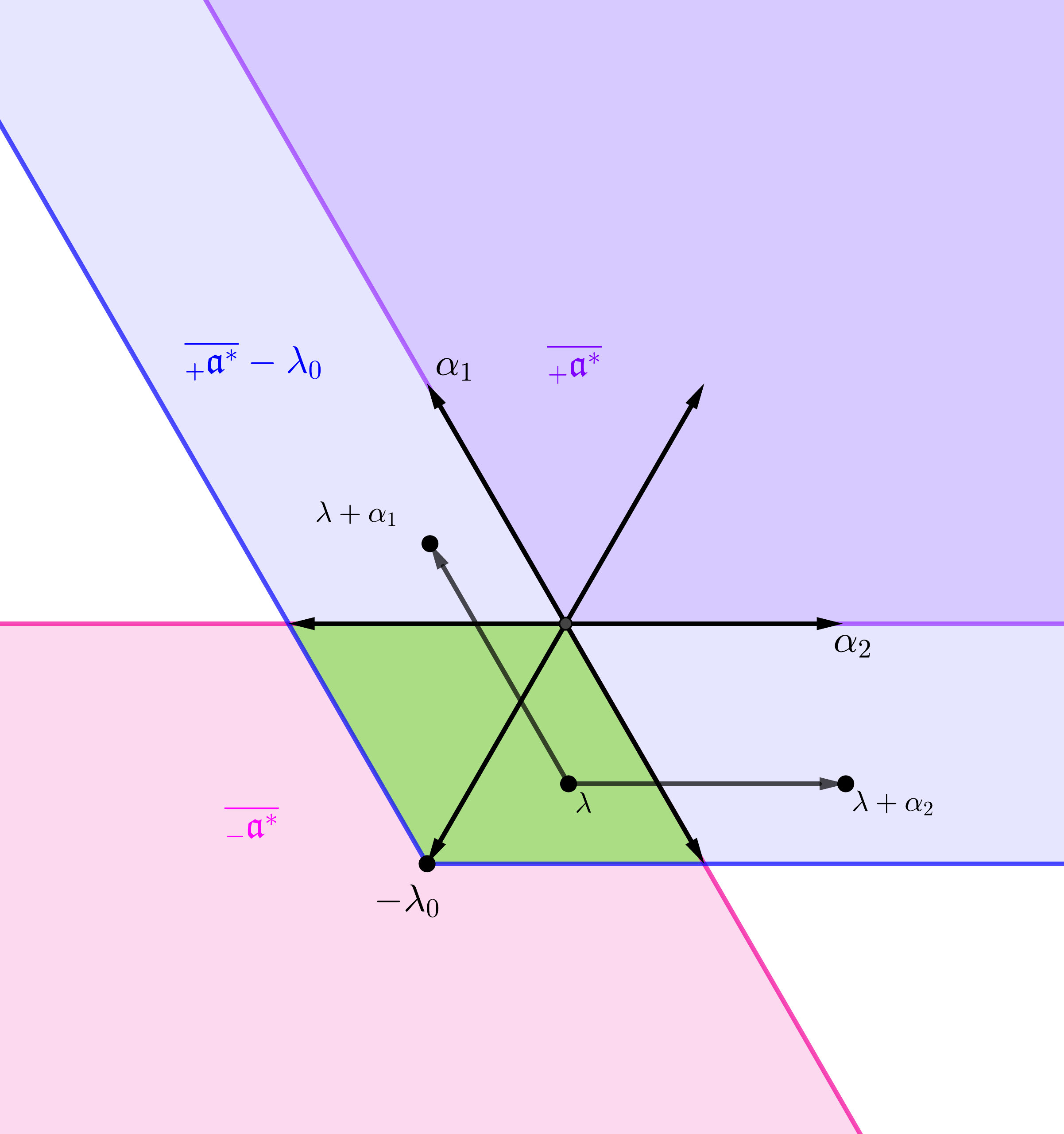}
 \caption{For $G=SL_3(\R)$ the green region depicts the real part of  the region where every resonance is a  first band resonance (see Proposition~\ref{prop:horocycle}).}
\end{figure}

\begin{proof}
 Since the horocycle operators are differential operators, we obtain a continuation to distributional sections and Lemma~\ref{la:commutationhoro} still holds. Let $u\in \res_{{}_\G \mathbf{X}^\tau}(\lambda)$, i.e. $u\in \mc D'(\mc M,\G\backslash \mc V_\tau)$ with $\WF(u)\subseteq E_u^\ast$ and ${}_\G \mathbf{X}^\tau_H u = -\lambda(H)u$. Since differential operators do not increase the wavefront set, we have $\WF(\mc U_{-\alpha}u)\subseteq E_u^\ast$. Furthermore, $${}_\G \mathbf{X}^{\tau\otimes \tau_{-\alpha}}_H \mc U_{-\alpha}u  = -\alpha(H)\mc U_{-\alpha} u + \mc U_{-\alpha} {}_\G \mathbf{X}^\tau_H u = -(\lambda + \alpha)(H) \mc U_{-\alpha} u$$ by Lemma~\ref{la:commutationhoro}. Hence $\mc U_{-\alpha}u \in \res_{{}_\G \mathbf{X}^{\tau\otimes \tau_{-\alpha}}}(\lambda + \alpha)$. 
 
 For the `in particular' part recall that $\res_{{}_\G \mathbf{X}^{\tau'}}(\lambda') = 0 $ for each unitary representation $\tau'$ of $M$ and $\Re(\lambda') \not\in \ov {{}_-\mf a^\ast}$ (see Proposition~\ref{prop:locationgeneral}) and $\res_{{}_\G \mathbf{X}^\tau}^0(\lambda) = \{u\in \res_{{}_\G \mathbf{X}^\tau} (\lambda)\mid \mc U_{-\alpha} u = 0 \,\, \forall \alpha \in \Sigma^+\}.$
 \end{proof}
 Note that $\bigcap_{\alpha\in \Pi} \ov {{}_-\mf a^\ast} \setminus (\ov {{}_-\mf a^\ast}-\alpha) = \ov {{}_-\mf a^\ast} \cap ({}_+ \mf a^\ast - \lambda_0)$,  where $\lambda_0 =\sum_{\alpha\in \Pi}\alpha$.
 Indeed, let $\lambda = \sum_{\alpha\in \Pi}c_\alpha \alpha \in \mf a^\ast.$ Then $\lambda\in \ov{{}_-\mf a^\ast}$ iff $c_\alpha\leq 0$ for all $\alpha\in \Pi$, $\lambda\in\ov {{}_-\mf a^\ast}-\alpha $ iff $c_\alpha \leq -1$ and $c_\beta\leq 0$ for all $\beta\in \Pi\setminus\{\alpha\}$, and $\lambda\in {}_+ \mf a^\ast $ iff $c_\alpha>0$ for all $\alpha\in \Pi$. Combining these statements implies the claim.

\subsection{First band resonant states and principal series representation}
In this section we  identify first band resonances states with certain $\G$-invariant vectors in a corresponding principal series representation. The proof follows the line of arguments given in \cite[Section 2]{KW17} in the rank one case.

By analogy to \cite[Definition 2.1]{KW17} we define 
\begin{align*}
	\mathcal R (\lambda)\coloneqq \{s\in \mc D'(G/M, \mc V_\tau) \mid (\mathbf{X}^\tau_H + \lambda(H))s = 0, \, \nabla_{\mf X_-}s=0 \, \forall \mf X_- \in C^\infty(G/M,G\times_{\Ad|_M} \mf n_-), H\in \mf a\}.
\end{align*}

\begin{lemma}\label{la:gammainv}
 The space $\res_{{}_\G \mathbf X^\tau}^0(\lambda)$ is isomorphic to the space of $\G$-invariants of $\mc R(\lambda)$, where the isomorphism is defined by considering $\G$-invariant sections as sections of the bundle $\G\backslash \mc V_\tau$.
\end{lemma}
\begin{proof}
The only part to observe is that each $s\in \mc R(\lambda)$  automatically has $\WF(s)\subseteq G\times_{\Ad|_M} \mf n^\ast$. This holds because $G\times_{\Ad^\ast|_M} \mf n^\ast$ is the joint characteristic set of $\mathbf{X}^\tau_H$ and $\mf X_-$ (see \cite[Lemma 2.5]{KW17} for details).
\end{proof}

We define $\Phi^\lambda \in C^\infty(G/M)$ by $\Phi^\lambda(gM) \coloneqq e^{-\lambda (H_-(g))}$ where $H_-(g)\in \mf a$ is defined via $g=k_-(g)\exp(H_-(g))n_-(g))$ with $k_-(g)\in K$ and $n_-(g)\in N_-$.

\begin{lemma}[see {\cite[(2.3)]{KW17}}]
	$$ X_H\Phi^\lambda = -\lambda(H) \Phi^\lambda \quad \text{and} \quad \mf X_-\Phi^\lambda=0$$
	for all $H\in \mf a$ and $\mf X_-\in C^\infty(G/M,G\times_{\Ad|_M} \mf n_-)$. Furthermore,
	$$\Phi^\lambda (\g gM) = e^{-\lambda (H_-(\g k_-(g)))}\Phi^\lambda(gM)$$
	for all $\g,g \in G$.
\end{lemma}

\begin{proof}
	We compute 
	\begin{align*}
		 X_H\Phi^\lambda(gM)  &= \d  e^{-\lambda (H_-(ge^{tH}))} =\d e^{-\lambda (H_-(g)+ tH)} = -\lambda(H) \Phi^\lambda(gM)
		\intertext{and}
		\mf X_-\Phi^\lambda(gM)  &= \d  e^{-\lambda (H_-(g\exp(t \ov{\mf X_-}(g))))} =\d e^{-\lambda H_-(g)} = 0.
	\end{align*}
	Let $g =kan_-$ and $\g k = k'a'n_-'$ be the opposite Iwasawa decompositions of $g$ and $\g k$, respectively. Then we have 
	\begin{equation*}
		H_-(\g g) = H_-(k'a'n_-'an_-) = \log a' + \log a = H_-(\g k_-(g)) + H_-(g)
	\end{equation*}
	and the second part of the lemma follows.
\end{proof}

The above lemma allows us to relate $\mc R(\lambda)$ to $\mc R(0)$.

\begin{lemma}[{\cite[Lemma 2.4]{KW17}}]\label{la:relationdifferentresonances}
	For each $\lambda \in \mf a_\C^\ast$ there is an isomorphism of topological vector spaces  (with the topology induced by $\mc D'(G/M;\mc V_\tau)$)
	\begin{align*}
	\mc R(\lambda) & \rightleftarrows \mc R(0)\\
	s &\mapsto \Phi^{-\lambda}s\\
	\Phi^\lambda s &\mapsfrom s.
	\end{align*}
\end{lemma}

\begin{proof}
	Fix $\lambda_1,\lambda_2 \in \mf a_\C^\ast$ and let $s \in \mc R(\lambda_1)$.
	We compute
	\begin{align}
		\mathbf X_H (\Phi^{\lambda_2}s) =  X_H(\Phi^{\lambda_2})s + \Phi^{\lambda_2} \mathbf{X}_H (s) = -(\lambda_2 +\lambda_1)(H) \Phi^{\lambda_2}s
		\intertext{and}
		\nabla_{\mf X_-}(\Phi^{\lambda_2}s)= \mf X_- (\Phi^{\lambda_2})s+ \Phi^{\lambda_2}	\nabla_{\mf X_-}s=0.
	\end{align}
	The claim follows with an appropriate choice of $(\lambda_1,\lambda_2)$.
\end{proof}

 Recall from Section~\ref{sec:principalreps} that  $\mc V_\tau^{\mc B}$ denotes the restriction of $\mc V_\tau$ to $K/M\subseteq G/M$.
\begin{lemma}[{see \cite[Lemma 2.6 and Prop. 2.8]{KW17}}]\label{la:identificationresonancezero}
	Let $\iota\colon K/M\to G/M$ be the inclusion and $B\colon G/M \to K/M, gM\mapsto k_-(g)M$ where $g=k_-(g)e^{H_-(g)}n_-(g)\in KAN_-$ is the opposite Iwasawa decomposition. Then $B$ and $\iota$ induce well-defined continuous pullback maps 
	\begin{align}
		B^\ast\colon \mc D'(K/M;\mc V_\tau^{\mc B}) \to \mc R(0),  \quad \iota^\ast \colon \mc R(0) \to  \mc D'(K/M;\mc V_\tau^{\mc B})
	\end{align}
	which are mutually inverse $G$-equivariant isomorphisms of topological vector spaces each equipped with the left regular $G$-action.
\end{lemma}

 We can now prove the main result of this section.
\begin{proposition}\label{prop:firstbandprinicipal}
	With the longest Weyl group element $w_0$ (see Section~\ref{sec:notation}) we have an isomorphism $$\res_{{}_\G\mathbf X^\tau}^0(\lambda) \to {}^\G (H^{w_0 \tau,w_0 (\lambda+\rho)})^{-\infty}$$ where ${}^\G (H^{w_0 \tau,w_0( \lambda+\rho)})^{-\infty}\subseteq \mc D'(K/M;\mc V_\tau^{\mc B})$ denotes the $\G$-invariant distributional vectors in the principal series representation $\pi_{w_0 \tau,w_0(\lambda+\rho)}$. 
\end{proposition}
\begin{proof}

 Let $Q_\lambda\colon \mc R(\lambda) \stackrel{\cdot \Phi^{-\lambda}}{\to} \mc R(0) \stackrel{\iota^\ast}{\to} \mc D'(K/M;\mc V_\tau^{\mc B})$ be the composition of the isomorphisms from Lemma~\ref{la:identificationresonancezero} and Lemma~\ref{la:relationdifferentresonances}. We will first prove that if $\mc D'(K/M;\mc V_\tau^{\mc B})$ is equipped with the principal series representation $\pi_{\tau,\lambda+\rho,\Sigma^-}$ (the principal series representation with respect to the positive system $\Sigma^- =-\Sigma^+$), then $Q_\lambda$ is $G$-equivariant. Note that $\rho$ is given by the positive system $\Sigma^+$.
  
Again we will identify sections of the bundles of $G/M$ and $K/M$ with $M$-equivariant functions. By \cite[Corollary~2.9]{KW17} it suffices to consider smooth sections $s\in \mc R(\lambda)$. Then we compute
	\begin{align*}
		\pi_{\tau,\lambda+\rho,\Sigma^-}(g)(\iota^\ast(\Phi^{-\lambda}s))(k)&=e^{-(\lambda + \rho + \rho_-)(H_-(g\inv k))}\iota^\ast(\Phi^{-\lambda}s(k_-(g\inv k)))\\
		&=e^{-\lambda (H_-(g\inv k))} (\Phi^{-\lambda}s)(k_-(g\inv k)),
		\intertext{and since $\Phi^{-\lambda}s\in \mc R(0)$ we get }
		&=e^{-\lambda (H_-(g\inv k))} (\Phi^{-\lambda}s)(g\inv k) = s(g\inv k).
	\end{align*} 
	But on the other hand we have
	\begin{align*}
		\iota^\ast(\Phi^{-\lambda}(g s))(k)=\Phi^{-\lambda}(k)(g s)(k)= s(g\inv k),
	\end{align*}
	so the $G$-equvariance of $Q_\lambda$ is proved. We get a $G$-invariant isomorphism $\mc R(\lambda) \to (H^{w_0 \tau,w_0( \lambda+\rho)})^{-\infty}$ since $\pi_{\tau,\lambda,\Sigma^-}$ and $\pi_{w_0\tau,w_0\lambda}$ are equivalent. Indeed,  $If(k)\coloneqq f(kw_0)$ for $f\colon K\to V_\tau$ with $f(km)=\tau(m\inv)f(k)$ defines an intertwiner between the compact pictures. Applying now Lemma~\ref{la:gammainv} completes the proof.
\end{proof}

\subsection{Quantum-classical correspondence}\label{sec:quantumclassical}
In the previous section we identified the first band resonant states $\res_{{}_\G\mathbf X^\tau}^0(\lambda)$ with $\G$-invariant distributional vectors in the principal series $(H^{w_0 \tau,w_0 (\lambda+\rho)})^{-\infty}$.    If we  restrict ourselves to the scalar case $\tau = 1$, then the Poisson transform $\mc P_{-w_0(\lambda+\rho)}$ defines a map from  ${}^\G (H^{w_0 (\lambda+\rho)})^{-\infty}$ to ${}^\G E_{-w_0(\lambda+\rho)}$, as  $\mc P_{-w_0(\lambda+\rho)}$ provides a $G$-invariant map $ (H^{w_0 (\lambda+\rho)})^{-\infty}$ to $ E_{-w_0(\lambda+\rho)}$ (see Section~\ref{sec:Poissontransform}). Hence, we can identify eigendistributions of the classical motion with quantum states and we call this identification \emph{quantum-classical correspondence}. More precisely, we have the following result, which immediately gives Theorem~\ref{thm:qcc}.
\begin{proposition}\label{prop:qcc}
 If $\lambda\in \mf a_\C^\ast$ satisfies $\frac {2\langle \lambda+\rho,\alpha\rangle}{\langle\alpha,\alpha\rangle}\not \in -\N_{>0}$ for all $\alpha\in \Sigma^+$, then we have a bijection $$\res_{{}_\G\mathbf X}^0(\lambda) \to {}^\G E_{-w_0(\lambda+\rho)} = {}^\G E_{-(\lambda+\rho)}.$$ In particular, $\lambda \in \sigma_{\RT}^0({}_\G\mathbf X)$ if and only if ${}^\G E_{-(\lambda+\rho)}\neq 0$. Furthermore, the isomorphism is given by the push-forward $\pi_\ast$ of distributions along the canonical projection $\pi:\G\backslash G/M\to \G\backslash G/K$.
\end{proposition}
\begin{proof}
 In view of Section~\ref{sec:Poissontransform} the Poisson transform is a bijection from $(H^{w_0(\lambda+\rho)})^{-\omega}\to E_{-\lambda-\rho}$. Restricted to $\G$-invariant distributional vectors it is still injective with image ${}^\G E_{-\lambda-\rho}$ since $\G$ is cocompact and therefore $\G$-invariant distributions have moderate growth (see \cite[Remark~12.5]{vdBanSchl87}). 
 
 It remains to show that the isomorphism is the push-forward along the canonical projection. To this end let $s\in \mc R(\lambda)$ be smooth and $p\colon G/M\to G/K$ the canonical projection. Then the isomorphism $\mc R(\lambda)\to (H^{w_0(\lambda+\rho)})^{-\infty}$ carries $s$ to $\tilde s\colon G\to \C, \tilde s(g)=s(gw_0)$. It follows that $$\mc P_{-w_0(\lambda+\rho)}\tilde s(gK)=\int _K \tilde s(gk)dk =\int_K s(gkw_0)dk=\int _K s(gk)dk$$ since $K$ is unimodular. On the other hand,  for $f\in C^\infty_c(G/K)$ we have $$(p_\ast s)(f)=s(f\circ p)=\int _{G/M} s(gM)f(gK)dgM = \int_{G/K}\left (\int_{K/M} s(gkM)dkM\right)f(gK)dgK$$
 if we normalize the Haar measure on $M$ and choose compatible invariant measures on $G/K$ and $K/M$.
 Hence, $p_\ast s= \mc P_{-w_0(\lambda+\rho)}\tilde s$ for $s\in \mc R(\lambda)\cap C^\infty(G/M)$. Using the density of smooth compactly supported functions in $\mc R(\lambda)$ \cite[Corollary~2.9]{KW17} we obtain the equality for the whole space $\mc R(\lambda)$. As before we now restrict to $\G$-invariant distributions identified with distributions on $\G\backslash G/M$ and $\G \backslash G/K$ to complete the proof.
\end{proof}

\section{Quantum spectrum}
In this section we analyze the quantum spectrum of the locally symmetric space $\G\backslash G /K$. Recall the definition of the joint eigenspace $$E_\lambda = \{f\in C^\infty (G/K) \mid Df = \chi_\lambda (D) f \quad \forall \, D\in \mathbb D(G/K)\}$$ for $\lambda \in \mf a_\C^\ast$. For the definition of $\chi_\lambda$ see Section~\ref{sec:notation}.
Since $D\in \mathbb D(G/K)$ is $G$-invariant, it descends to a differential operator ${}_\G D$ on the locally symmetric space $\G\backslash G/K$. Therefore, the left $\G$-invariant functions of $E_\lambda$ (denoted by ${}^\G E_\lambda$) can be identified with joint eigenfunctions on $\G\backslash G/K$ for each ${}_\G D$: $${}^\G E_\lambda = \{f\in C^\infty(\G\backslash G /K)\mid {}_\G D f = \chi_\lambda (D) f \quad \forall \, D\in \mathbb D(G/K)\}.$$ This leads to the following definition.

\begin{definition}\label{def:quantum spectrum}
 The \emph{quantum spectrum} of $\G\backslash G /K$ is defined as $$\sigma_Q \coloneqq \sigma_Q(\G\backslash G/K)\coloneqq \{\lambda\in \mf a_\C^\ast\mid {}^\G E_\lambda\neq 0\}.$$
\end{definition}
Using the quantum-classical correspondence and the Weyl law from \cite{dkv} we can now prove Theorem~\ref{thm:Weyl}.
\begin{proof}[Proof of Theorem~\ref{thm:Weyl}]
 From \cite[Theorem~8.9]{dkv} we have for each set $\Omega\subset \mf a^\ast$ as in Theorem~\ref{thm:Weyl}
 $$\sum_{\lambda\in \sigma_Q\cap i\mf a^\ast\!, \Im \lambda\in t \Omega} \dim ({}^\G E_\lambda) |W\lambda|\inv =\textup{Vol}(\G\backslash G/K) \left(2\pi \right)^{-d} \textup{Vol}(\Ad(K)\Omega) t^d +\mc O(t^{d-1}),$$ where $\textup{Vol}(\G\backslash G/K)$ is the volume of the compact Riemannian manifold $\G\backslash G/K$ with Riemannian structure induced by the Killing form and  $\textup{Vol}(\Ad(K)\Omega)$ is the volume of the set $\Ad(K)\Omega\subseteq \Ad(K)\mf a$ with respect to the Killing form restricted to $\Ad(K)\mf a$. 
 Replacing $\Omega$ by $\Omega\setminus \bigcup_{\alpha\in \Sigma ^+}\alpha^\perp$ we deduce $\sum_{\lambda\in \sigma_Q\cap i\mf a^\ast\!, \Im \lambda\in t \Omega\cap \bigcup \alpha^\perp} \dim ({}^\G E_\lambda) =\mc O(t  ^{d-1})$ since $\textup{Vol}(\Ad(K)\alpha^\perp)=0$. Therefore,
 $$\sum_{\lambda\in \sigma_Q\cap i\mf a^\ast\!, \Im \lambda\in t \Omega} \dim ({}^\G E_\lambda) =|W|\textup{Vol}(\G\backslash G/K) \left(2\pi \right)^{-d} \textup{Vol}(\Ad(K)\Omega) t^d +\mc O(t^{d-1})$$ since $W$ acts freely on the Weyl chambers. To complete the proof we observe that $\sigma_{\RT}({}_\G\mathbf X)\supseteq \sigma_{\RT}^0({}_\G\mathbf X)$ and $m(\lambda)\geq \dim(\res_{{}_\G \mathbf{X}}^0(\lambda))=\dim ({}^\G E_{-\lambda-\rho})$ for $\lambda \in i\mf a^\ast$.
\end{proof}

As $\chi_\lambda = \chi_{w\lambda}$ for $w\in W$ it is obvious that $\sigma_Q$ is $W$-invariant.
The following properties of $\sigma_Q$ were derived by Duistermaat-Kolk-Varadarajan \cite{dkv}. We include the proof for the convenience of the reader.
\begin{proposition}[{see \cite[Prop.~2.4,~Prop.~3.4,~Cor.~3.5]{dkv}}]\label{prop:quantumimpliespositive}
 If $\lambda\in \sigma_Q$, then the corresponding spherical function is positive semidefinite. Moreover, there is some $w\in W$ such that $w\lambda = -\ov \lambda$ and $\Re \lambda \in \conv (W\rho)$. In particular, $\langle\Re \lambda,\Im \lambda\rangle =0$ and $\norm{\Re \lambda}\leq \norm \rho$.
\end{proposition}
\begin{proof}
 Pick $u\in {}^\G E_{\lambda}$, regarded as a right $K$-invariant element of $L^2(\G\backslash G)$, normalized such that $\langle u,u\rangle_{L^2(\G\backslash G)} =1$. With the right regular representation $R$ on $L^2(\G\backslash G)$ define $\Phi(g)\coloneqq \langle R(g)u,u\rangle$. 
 Being a matrix coefficient the function $\Phi$ is positive semidefinite.
 We will show that $\Phi$ is the elementary spherical function $\phi_\lambda$. 
 By right $K$-invariance of $u$ and unitarity of $R$ we get that $\Phi$ is $K$-biinvariant. $\Phi(1)=1$ is obvious. Smoothness follows from the fact that $u$ is smooth. 
 Furthermore, $D \Phi (g) =\langle R(g)Du,u\rangle = \chi_{\lambda}(D) \Phi(g)$ by left invariance of $D$. 
 We conclude that $\Phi$ is the elementary spherical function for $\chi_{\lambda}$, i.e. $\Phi=\phi_\lambda$.
 
 Since $\phi_\lambda$ is positive semidefinite we have $\phi_\lambda (g)= \ov{\phi_\lambda(g\inv)}$ by definition of positive definiteness and $\ov{\phi_\lambda(g\inv) } =\phi_{-\ov\lambda}(g)$ by the integral representation (see Section~\ref{sec:sphericalpos}). Therefore $\phi_{\lambda}=\phi_{-\ov \lambda}$ implying that $w\lambda = -\ov \lambda$ for some $w\in W$. It easily follows that $$\langle \Re\lambda, \Im\lambda \rangle = \langle w \Re\lambda,w\Im\lambda\rangle = \langle -\Re\lambda ,\Im\lambda\rangle = 0.$$
 
 Moreover, $\phi_\lambda$ is bounded which holds iff $\Re\lambda \in \conv (W\rho)$ (see \cite[Ch. IV Theorem~8.1]{gaga}). Since $\{\mu \in \mf a^\ast\mid \norm{\mu}\leq \norm\rho\}$ is convex and contains $W\rho$, the last assertion follows.
\end{proof}

\begin{remark}
 In the rank one case Proposition~\ref{prop:quantumimpliespositive} implies for $\lambda\in \sigma_Q$ that $\lambda\in \mf a^\ast$ with $\|\lambda\|\leq\|\rho\|$ or $\lambda\in i\mf a^\ast$. In this particular case, this can be obtained not only from  Proposition~\ref{prop:quantumimpliespositive} but also from the positivity of the Laplacian on $\G\backslash G/K$. In the higher rank setting the algebra $\mathbb D(G/K)$ contains more operators, more precisely it is a polynomial algebra in $n$ variables. Using the properties of the Harish-Chandra isomorphism $\HC$ one can obtain that $-\ov\lambda\in W\lambda$ from the self/skew-adjointness of the operators in $\mathbb D(G/K)$.
\end{remark}

\begin{remark}\label{rmk:obstructions}
 Proposition~\ref{prop:quantumimpliespositive} implies the following obstructions for $\lambda \in \mf a_\C^\ast$ to be in $\sigma_Q$.
 \begin{enumerate}
  \item If $\Re \lambda = 0$, then we get no obstructions on $\Im \lambda$ since $w\lambda = -\ov\lambda$ is satisfied with $w=1$.
  \item If $\Re \lambda \neq 0$, then $\Im \lambda$ is singular, i.e.  $\Im \lambda\in \alpha^\perp$ for some $\alpha\in \Sigma$, since  $\Im \lambda$ non-singular implies $w=1$ as $W$ acts simply transitively on open Weyl chambers.
  \item If $\Re \lambda$ is regular, i.e. $\langle \Re\lambda,\alpha\rangle \neq 0$ for all $\alpha\in \Sigma$,  we denote by $\tilde w_0$ the unique Weyl group element mapping the Weyl chamber containing $\Re\lambda$ to its negative. Then 
   we have $\lambda \in \operatorname{Eig}_{-1}(\tilde w_0)+ i\operatorname{Eig}_{+1}(\tilde w_0)\subseteq \mf a_\C^\ast$ where  $\operatorname{Eig}_{\pm 1}$ denotes the eigenspace for $\pm 1$.
   If $-1$ is contained in $W$, then $\Im \lambda =0$. In particular, this is true in the rank one case but need not hold in general as is seen below. 
 \end{enumerate}
 Let us calculate $\dim \operatorname{Eig}_{+1}(w_0)=\dim \operatorname{Eig}_{+1}(\tilde w_0)$ in order to control the amount of freedom for $\Im\lambda$.  Let $d_\pm\coloneqq \dim \operatorname{Eig}_{\pm 1}(w_0)$. Then $n=d_++d_-$ and $\operatorname{Tr}(w_0)=d_+-d_- $. Choosing the basis $\Pi$ we observe $\operatorname{Tr}(w_0)= -\# \{\alpha\in \Pi\mid w_0\alpha = -\alpha\}\leq 0$. Thus, $d_\pm=\frac 12 (n\pm\operatorname{Tr}(w_0))$ so that $d_+ \leq \frac n 2$. We obtain the following traces and dimensions for the irreducible root systems from the classification.

 \begin{figure}[h]
		\centering
	\begin{tabular}{c|cccccccccc}
Type &$A_n$, $n$ even &$A_n$, $n$ odd&$B_n$&$D_n$, $n$ even &$D_n$, $n$ odd &$E_6$&$E_7$&$E_8$&$F_4$&$G_2$ \\\hline
 $-\operatorname{Tr} (w_0) $ & 0&1&$n$&$n$&$n-2$&2&7&8&4&2\\
 $d_+$& $n/2$&$(n-1)/2$&0&0&1&2&0&0&0&0
	\end{tabular}
\end{figure}

\end{remark}

\begin{example}
 For $G =SL_n(\R)$ an element $\lambda\in\mf a^\ast\simeq \mf a$ is regular iff the diagonal entries are pairwise distinct. For $\lambda =\operatorname{diag}(\lambda_1,\ldots,\lambda_n)\in \sigma_Q$ with $\Re\lambda\in  \ov{\mf a^\ast_+}$ satisfies $\Re \lambda_k = -\Re\lambda_{n-k}$ and  $\Im\lambda_k =\Im\lambda_{n-k}$ for all $k$ since the longest Weyl group element is the permutation $(1 \leftrightarrow n)(2\leftrightarrow n-1)\cdots$.
 
 More specifically, for $G=SL_3(\R)$ the only Weyl group elements with eigenvalue $-1$ are the reflections at hyperplanes perpendicular to the roots. Hence, $\lambda\in \sigma_Q$ implies $\Re\lambda\in[-1,1]\alpha$ and $\Im\lambda \in \alpha^\perp$ for some $\alpha\in \Sigma$ or $\lambda\in i\mf a^\ast$. The obstructions on $\lambda$ to be in $\sigma_Q$ described by Remark~\ref{rmk:obstructions} are less concrete and are visualized in Figure~\ref{fig:obstructions}.
\end{example}

\begin{figure}
 \centering
 \begin{minipage}[b]{0.4\linewidth}
  \includegraphics[width=\linewidth,trim = 130pt 160pt 130pt 150pt,clip]{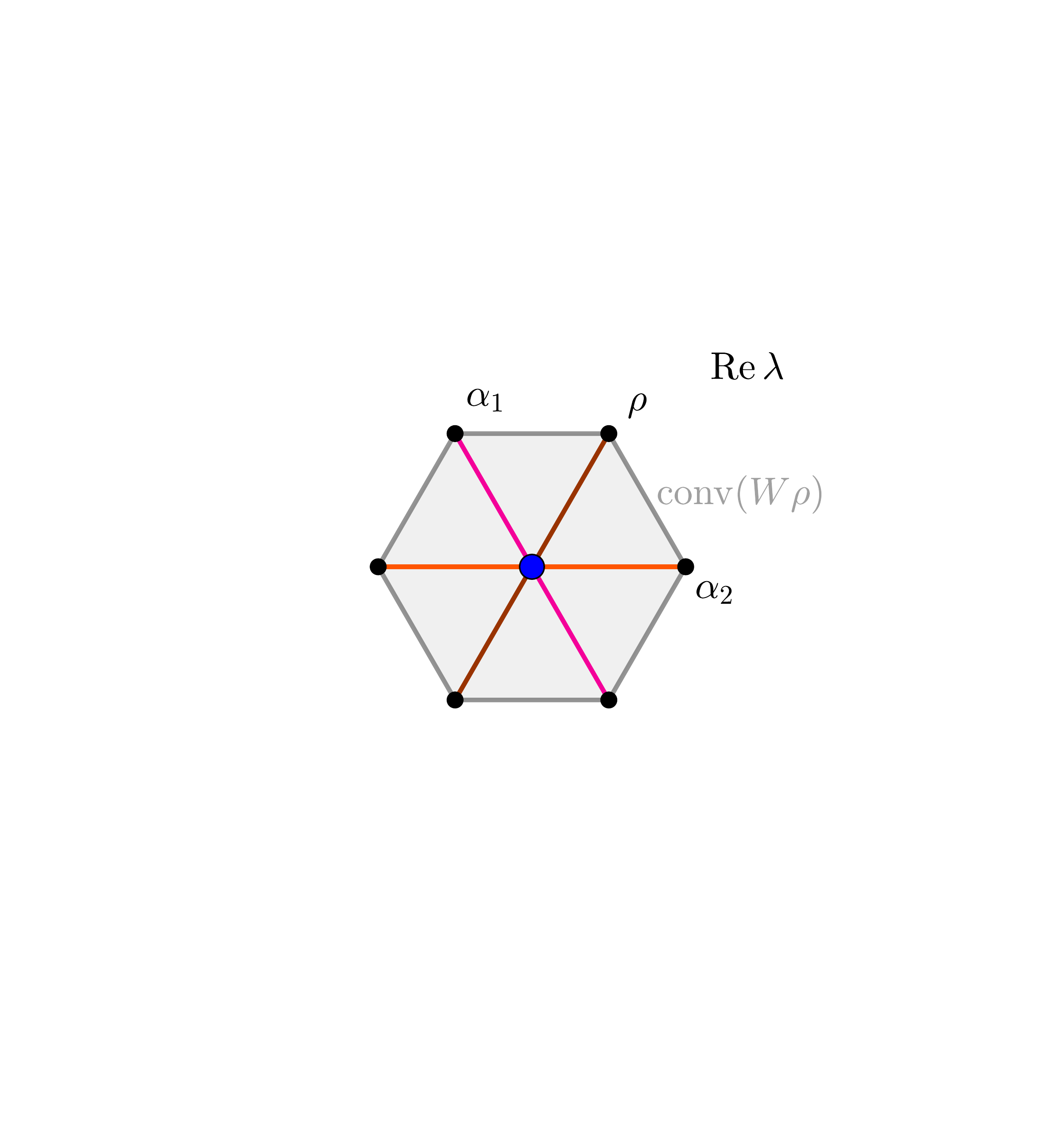}
\end{minipage}
\begin{minipage}[b]{0.4\linewidth}
  \includegraphics[width=\linewidth,trim = 130pt 160pt 130pt 150pt,clip]{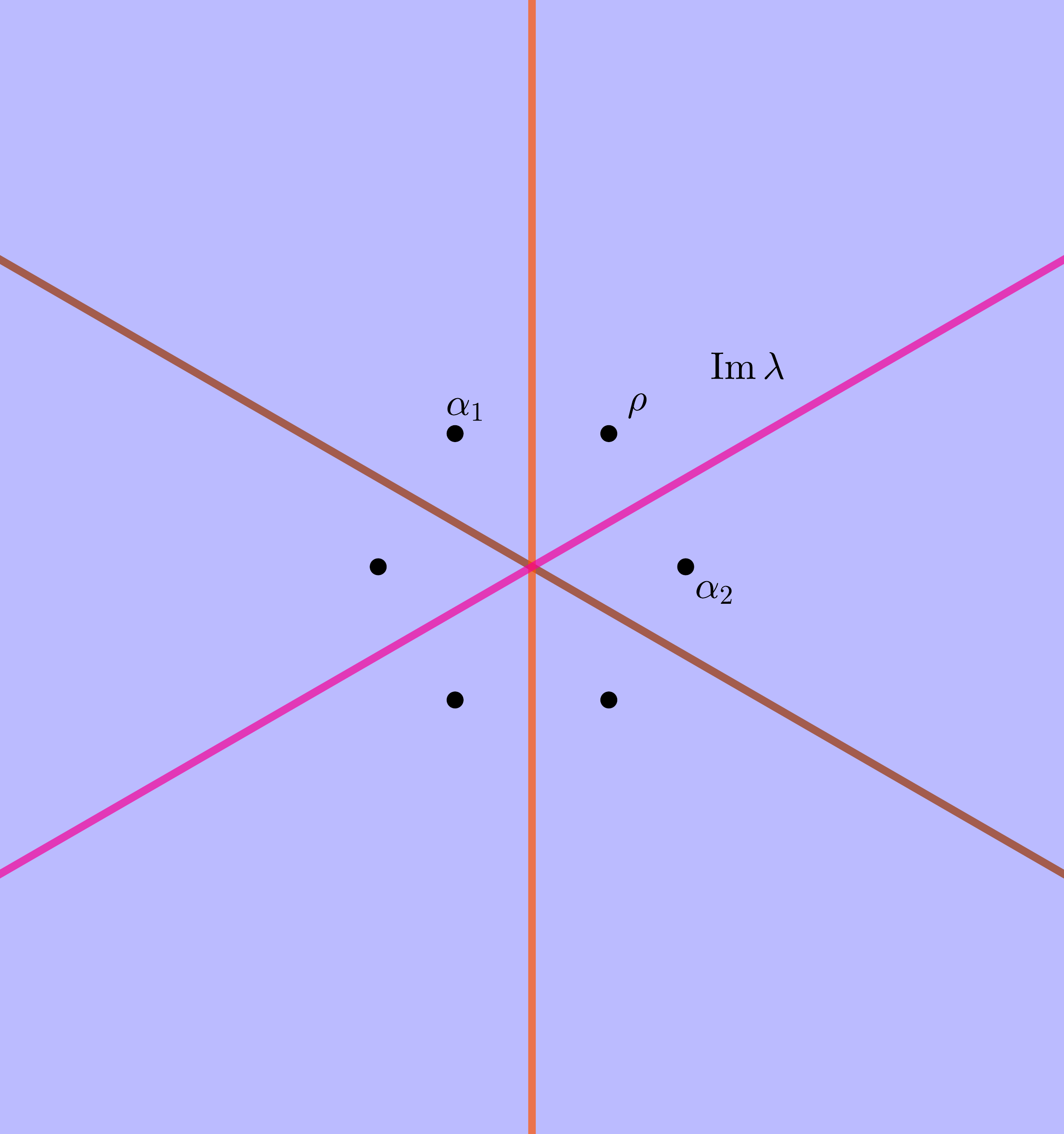}
\end{minipage}
 
\caption{Situation for $SL_3(\R)$ as obtained from Remark~\ref{rmk:obstructions}: If $\lambda\in \sigma_Q$ then $\Re\lambda$ is either equal to zero (blue dot in the left picture) or lies on one of the purple, orange or brown lines depicted on the left. Furthermore $\Im\lambda$ has to lie in the respective region depicted on the right, i.e. if $\Re\lambda=0$, $\Im\lambda$ can take any value (blue shaded plane), if $\Re\lambda$ lies on the orange line, then $\Im\lambda$ has to lie on the orange line and so on.}
\label{fig:obstructions}
\end{figure}

Let us formulate the condition that $\phi_\lambda$ is positive semidefinite in a different way. 

\begin{proposition}
 $\phi_\lambda$ is positive semidefinite if and only if the subrepresentation generated by the $K$-invariant vector in the  principal series representation $H^{w\lambda}$ is unitarizable and  irreducible for some $w\in W$. Equivalently, $H^{-w\ov \lambda}$ has a unitarizable irreducible spherical quotient.
\end{proposition}
\begin{proof}
 By Casselman's embedding theorem $\pi_{\phi_\lambda}$ is a subrepresentation of $H^{\tau,\nu}$ for some $\tau \in \widehat M$ and $\nu\in \mf a_\C^\ast$ (see e.g. \cite[Theorem~8.37]{Kna86}). More precisely, the $(\mf g, K)$-module of $K$-finite vectors are equivalent. Since the only principal series representations containing $K$-invariant vectors are the spherical ones, we obtain $\tau=1$. Since infinitesimally equivalent admissible representations of $G$ have the same set of $K$-finite matrix coefficients (see \cite[Corollary~8.8]{Kna86}), we conclude $\phi_\lambda =\phi_\nu$, i.e. $w\lambda = \nu$.
 
 Conversely assume that the subrepresentation generated by the $K$-invariant vector in the  principal series representation $H^{w\lambda}$ is unitarizable and  irreducible. Again by \cite[Corollary~8.8]{Kna86} the matrix coefficient $\phi_{w\lambda}=\phi_\lambda$ of $H^{w\lambda}$ is a matrix coefficient of the  unitary representation obtained by the unitary structure as well. Hence, $\phi_\lambda$ is positive semidefinite.
 Transition to the dual representation implies the second equivalence.
\end{proof}

\begin{remark}
 Although the unitary dual is classified for many groups, it is difficult to deduce which elementary spherical functions are positive semidefinite. This is due to the fact that most classifications are not obtained in terms of quotients of the spherical principal series but use different descriptions of admissible representations.  However, for rank one groups everything is classified (see \cite[{p.484}]{gaga}): If $\alpha$ denotes the unique reduced root in $\Sigma ^ +$, then $\phi_\lambda$ is positive semidefinite iff $\lambda\in i\mf a^\ast$ or $\lambda \in \mf a ^\ast$ and $|\langle \lambda ,\alpha\rangle|\leq \langle \rho,\alpha\rangle$ for $2\alpha\not\in \Sigma$ (i.e. in the real hyperbolic case) and $|\langle \lambda,\alpha\rangle|\leq (m_\alpha/2 +1)\langle \alpha,\alpha\rangle$ for $2\alpha\in \Sigma$ or $\lambda =\pm \rho$. 
\end{remark}

\begin{figure}[ht]
\centering
\begin{minipage}[b]{0.4\linewidth}
 \includegraphics[width=\linewidth,trim=0pt 180pt 0pt 0pt,clip]{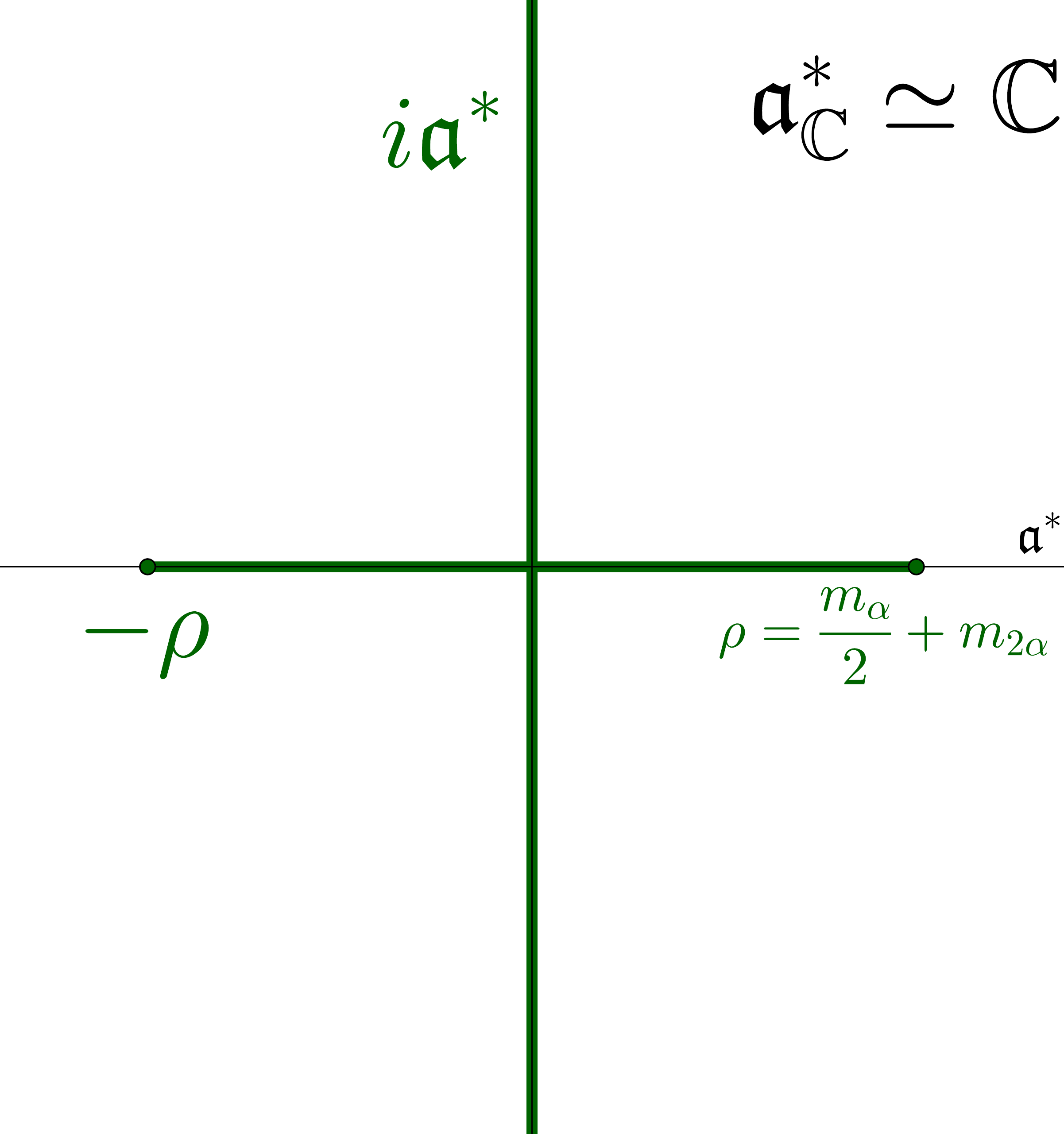}
\end{minipage}\qquad
\begin{minipage}[b]{0.4\linewidth}
 \includegraphics[width=\linewidth,trim=0pt 180pt 0pt 0pt,clip]{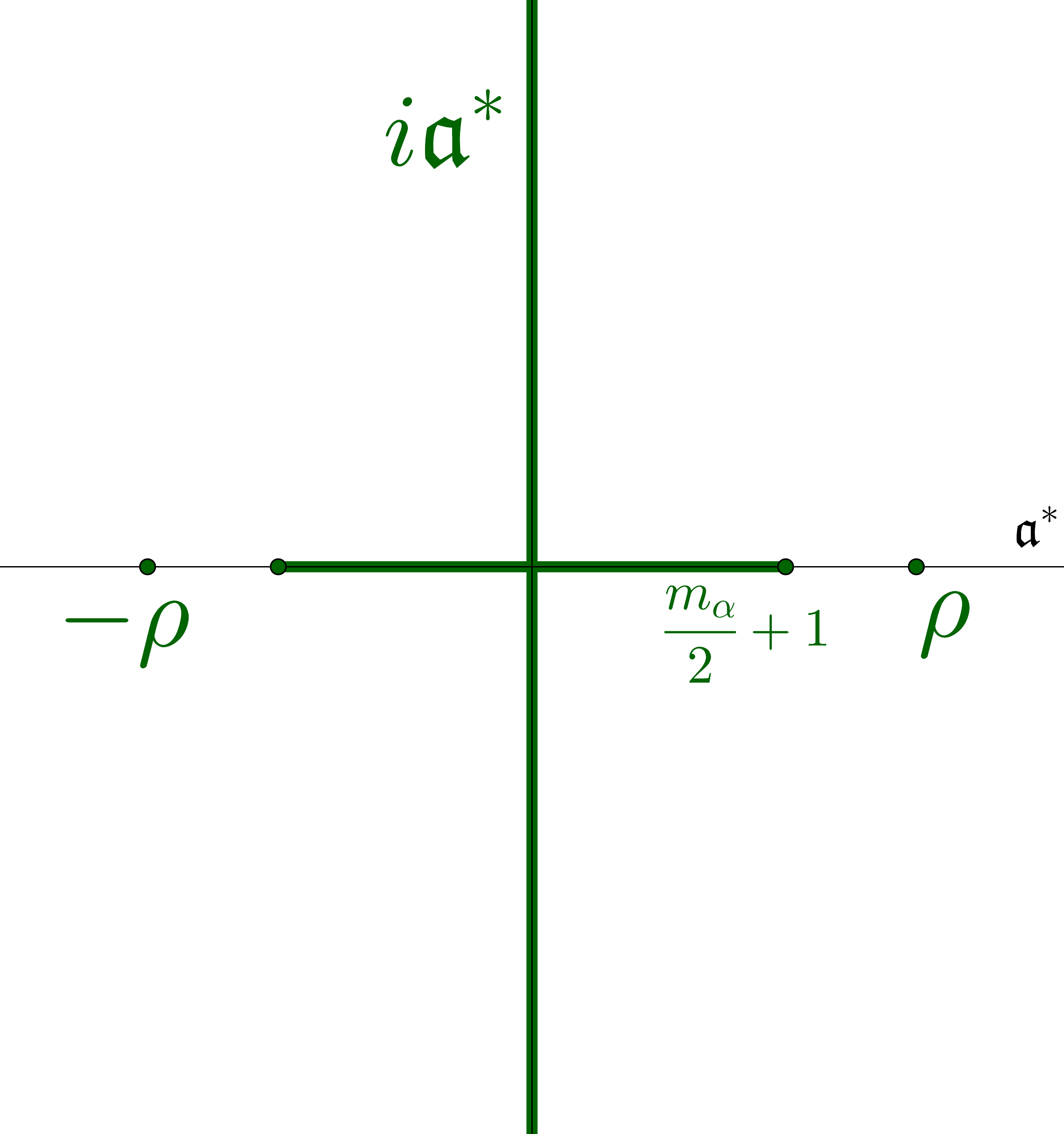}
\end{minipage}
\caption{Spherical dual in the rank one case. The picture on the left describes the real and complex hyperbolic case $m_{2\alpha}\leq 1$. The  picture on the right describes the quaternionic case $m_{2\alpha}\geq 2$. In the latter case note that   there is a spectral gap separating $\rho$.}
 
\end{figure}

\subsection{Property (T)}\label{sec:propT}
In this section we  review some facts about Kazhdan's Property (T) which will lead to a more precise description of the location of $\sigma_Q$. Recall that a locally compact group has \emph{Property (T)} iff the trivial representation is an isolated point in the unitary dual of the group with respect to the Fell topology (see \cite{propT} for a general reference). It is well known that each real simple Lie group of real rank $\geq 2$ has Property (T) (see \cite[Theorem~1.6.1]{propT}). Since the mapping $\lambda \mapsto \phi_\lambda$ is continuous and the correspondence between positive semidefinite elementary spherical functions and irreducible unitary spherical representations is a homeomorphism (see Section~\ref{sec:sphericalpos}), we obtain that in some neighbourhood of $\rho$ no elementary spherical function is positive semidefinite. We will use a more quantitative description introduced by  Oh \cite[Section~7.1]{oh2002}. Therefore, we denote by $p_K(G)$ the smallest real number such that the $K$-finite matrix coefficients of $\pi$ are  in $L^q(G)$ for any $q>p_K(G)$ and nontrivial $\pi\in \widehat G$.

\begin{remark}
 \begin{enumerate}
  \item Since each matrix coefficient  of $\pi \in \widehat G$ is bounded, it is contained in  $L^q$ for each $q>p$ if it is in $L^p$. Hence, $$p_K(G) =  \inf \{p \mid \text{all $K$-finite matrix coefficients of $\pi$ are in $L^p(G)$}  \quad \forall \pi \in \widehat G \setminus \{1\}\}.$$
  \item $p_K(G)\geq 2$.
  \item  By \cite{cowling} together with \cite{oh2002} we have $p_K(G)<\infty$ iff $G$ has Property (T).
 \end{enumerate}
\end{remark}
In many examples one knows the number $p_K(G)$  explicitly  or at least upper bounds.

\begin{example}[{see \cite[Section~7]{oh2002}}]

 \begin{enumerate}
 
  \item $p_K(SL_n(k))=2(n-1)$ for $n\geq 3$ and $k=\R,\C$.
  \item $p_K(Sp_{2n}(\R))= 2n$ for $n\geq 2$.
  \item $p_K(G)$ is bounded above by an explicit value for split classical groups of higher rank. 
 \end{enumerate}

\end{example}

We can now prove the following theorems.
\begin{theorem}\label{thm:locationquantum}
 Let $G$ be a non-compact real semisimple Lie group with finite center and $\G\leq G$ a discrete, cocompact, torsion-free subgroup. Then $$\Re \sigma_Q(\G\backslash G/K)\subseteq (1-2p_K(G)\inv)\conv(W\rho) \cup W\rho.$$
\end{theorem}
\begin{proof}
 Let $\lambda\in \sigma_Q(\G\backslash G/K)$. Then $\phi_\lambda$ is a matrix coefficient of the irreducible unitary representation $\pi_{\phi_\lambda}$. By the definition of $p_K(G)$ we have $\phi_\lambda\in L^{p_K(G)+\epsilon}(G)$ for all $\epsilon>0$ or $\pi_{\phi_\lambda}$ is the trivial representation. By Proposition~\ref{prop:Lpequivalence} we get $\Re\lambda\in (1-2p_K(G)\inv)\conv(W\rho)$ in the first case. The latter case occurs iff $\phi_\lambda\equiv 1$, i.e. $\lambda\in W\rho$. 
\end{proof}

\begin{theorem}\label{thm:gapquantum}
  Let $G$ be a non-compact real semisimple Lie group with finite center and $\G\leq G$ a discrete, cocompact, torsion-free subgroup. Then there is a neighborhood $\mc G$ of $\rho$ in $\mf a^\ast$ such that $$\sigma_Q(\G\backslash G/K)\cap (\mc G\times i\mf a^\ast)=\{\rho\}.$$
\end{theorem}
\begin{proof}
 Without loss of generality we assume that $G$ has trivial center, otherwise replace $G$ by $G/Z(G)$. Then $G$ is a product of  simple Lie groups $G_1,\ldots,G_l$ such that $G_1,\ldots,G_k$, $k\leq l$, are of rank one. With the obvious notation let $\lambda=(\lambda_1,\ldots,\lambda_l)\in (\mf a_1)_\C^\ast\oplus\cdots\oplus (\mf  a_l)_\C^\ast$ be in $\sigma_Q$. By Proposition~\ref{prop:locationgeneral} we have $w\lambda=-\ov \lambda$ for some $w\in W$. Hence $\lambda_i\in \mf a_i^\ast$ are real for $i\leq k$ if $\Re \lambda_i\neq 0$ and      $\Re \lambda_i\in  (1-2p_K(G_i)\inv)\conv(W_i\rho_i) \cup W_i\rho_i$ for $i>k$ by Theorem~\ref{thm:locationquantum}. Since $G_i$, $i>k$, have Property (T) we conclude that there is a neighborhood $U$ of $\rho$ in $\mf a ^\ast$ such that 
 $$\sigma_Q\cap(U\times i\mf a^\ast) \subseteq \mf a_1  ^\ast \times \cdots \times \mf a_k ^\ast\times \{\rho_{k+1}\}\times\cdots\times\{\rho_l\}.$$
 Discreteness of $\sigma_Q$ implies  the theorem.
\end{proof}

\section{Main Theorem}
In this section we present the main theorem of the article and deduce Theorem~\ref{thm:gap} from it.
\begin{theorem}\label{thm:main}
Let $G$ be a non-compact real semisimple Lie group with finite center and $\G\leq G$ a discrete, cocompact, torsion-free subgroup.
 Define $\mc A\coloneqq \{\lambda\in \mf a_\C^\ast\mid \frac {2\langle \lambda+\rho,\alpha\rangle}{\langle\alpha,\alpha\rangle} \in -\N_{>0}$ for some $\alpha\in \Sigma^+\}$, $\mc B\coloneqq \{\lambda\in \mf a_\C^\ast\mid w\lambda =-\ov \lambda $ for some $w\in W\}$, and $\mc F\coloneqq \{\lambda\in \mf a^\ast \mid  \lambda +\alpha \not\in  \ov{{}_-\mf a ^\ast} $ for all $\alpha \in \Pi\}$.
 Then we have the following inclusions $$\sigma_{\RT}({}_\G \mathbf{X})\cap (\mc F\times i\mf a^\ast) \subseteq \sigma_{\RT}^0( {}_\G \mathbf X) $$ and 
 \begin{align*}
  \sigma_{\RT}^0( {}_\G \mathbf X) \cap (\mf a_\C^\ast\setminus\mc A) &\subseteq -\sigma_Q(\G\backslash G/K) - \rho \\&\subseteq \mc B \cap (((1-2p_K(G)\inv)\conv(W\rho) \cup W\rho) +i\mf a^\ast)-\rho .
 \end{align*}

 \end{theorem}
\begin{proof}
 This is immediate from Propositions \ref{prop:horocycle}, \ref{prop:qcc}, and \ref{prop:quantumimpliespositive} and Theorem \ref{thm:locationquantum}.
\end{proof}

\begin{proof}[Proof of Theorem~\ref{thm:gap}]
 It follows from Theorem~\ref{thm:main} that the neighborhood can be chosen as $(\mf a_+ ^\ast-\rho)\cap \mc F\cap (-\mc G - \rho)$ where $\mc G$ is obtained by Theorem~\ref{thm:gapquantum}. If $G$ has Property (T), then $p_K(G)$ is finite and $\mc G$ can be replaced by the complement of the $\G$-independent set $(1-2p_K(G)\inv)\conv(W\rho)$.
\end{proof}

\begin{figure}[ht]
\centering
 \includegraphics[height=10cm,trim = 0pt 30pt  0pt 70pt,clip]{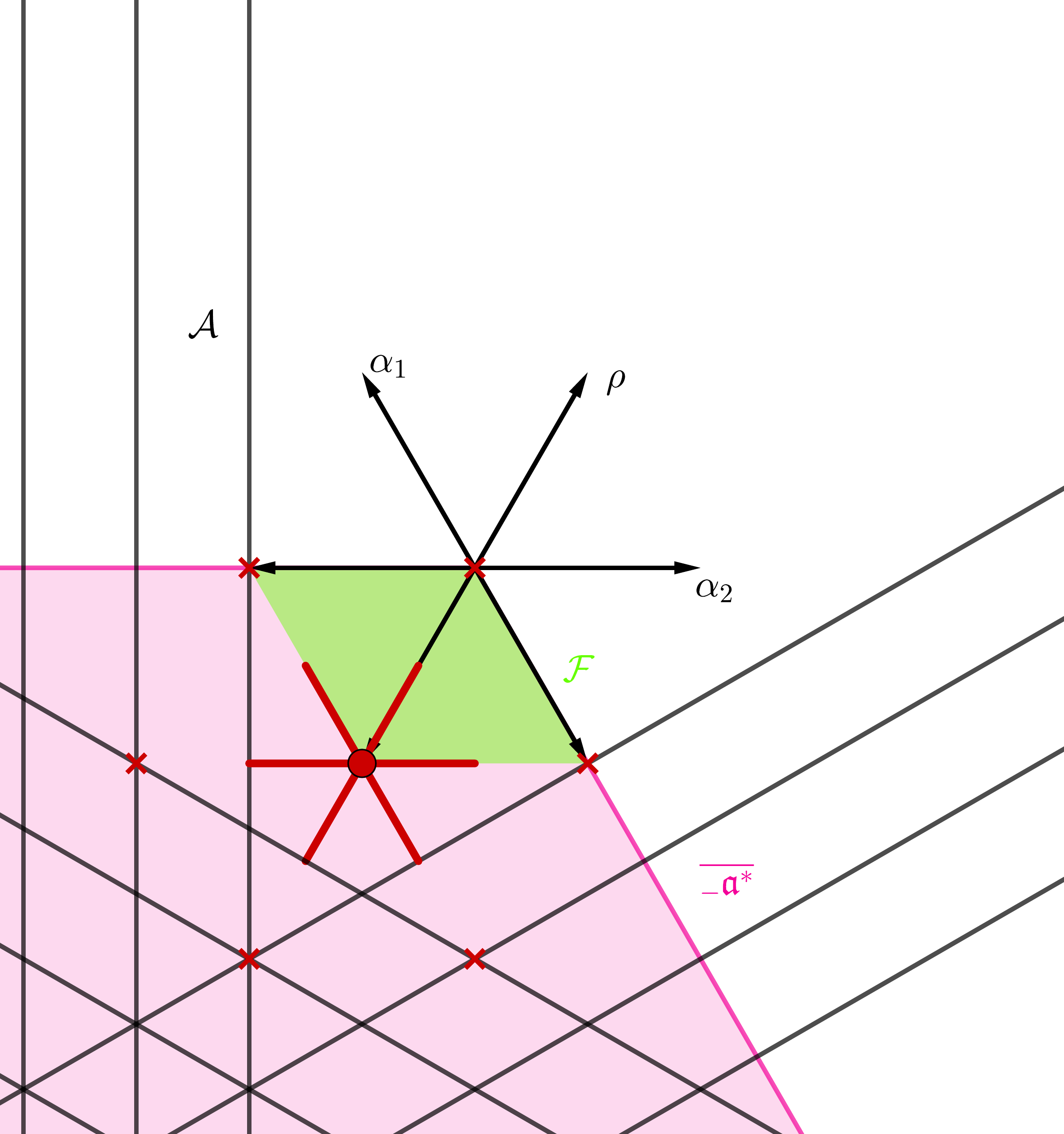}
\caption{Visualization of the real part of $\mf a_\C^\ast$ for $G=SL_3(\R)$: The pink region is the region where Ruelle-Taylor resonances can a priori be located in view of  the results of \cite{higherrank}. The red points and lines depict the region $(\mc B\cap \frac 12 \conv(W\rho)\cup W\rho)-\rho$, i.e. the region where first band resonances can occur. The green shaded region illustrates the real parts in which only first band resonances can occur. Further first band resonances might occur inside the exceptional set $\mathcal A$ depicted by the black lines.}\label{fig:first_band_sl3}
\end{figure}

\newpage\bibliographystyle{alpha}
\bibliography{literatur}

\bigskip
\bigskip

\end{document}